\documentclass[11pt,a4paper]{article}

\usepackage[utf8]{inputenc}
\usepackage[T1]{fontenc}
\usepackage[english]{babel}
\usepackage{lmodern}
\usepackage{microtype}
\usepackage{amsmath,amssymb,amsthm,mathtools}
\usepackage{geometry}
\usepackage{enumitem}
\usepackage{booktabs}
\usepackage{xcolor}
\usepackage{hyperref}
\usepackage{comment}
\hypersetup{colorlinks=true,linkcolor=blue!60!black,citecolor=blue!60!black,urlcolor=blue!60!black}
\setlist[itemize]{leftmargin=2.1em}
\setlist[enumerate]{leftmargin=2.1em}
\numberwithin{equation}{section}
\newtheorem{theorem}{Theorem}[section]
\newtheorem{proposition}[theorem]{Proposition}
\newtheorem{lemma}[theorem]{Lemma}
\newtheorem{corollary}[theorem]{Corollary}
\newtheorem{definition}[theorem]{Definition}
\newtheorem{assumption}[theorem]{Assumption}
\newtheorem{remark}[theorem]{Remark}

\newcommand{\R}{\mathbb R}
\newcommand{\N}{\mathbb N}

\newcommand{\Sscr}{\mathcal S}
\newcommand{\Dscr}{\mathcal D}
\newcommand{\Bsp}{\mathcal B_{\mathrm{sp}}}

\newcommand{\dd}{\,\mathrm d}
\newcommand{\eps}{\varepsilon}
\newcommand{\ang}[1]{\langle #1\rangle}

\newcommand{\one}{\mathbf 1}

\title{Cut-off Jastrow Factors and Spectral Barron Regularity of Coulombic Electronic Wave Functions}
\author{Virginie Ehrlacher}
\date{\today}

\begin{document}
\maketitle


\begin{abstract}
We study the spectral Barron regularity of Coulombic electronic
eigenfunctions after extraction of a cut-off Jastrow factor. Let
\(H=-\Delta+V\) be an \(N\)-electron Coulomb Hamiltonian with clamped
nuclei, and let \(\psi\) be an eigenfunction associated with a discrete
eigenvalue below the bottom of the essential spectrum. For the cut-off
Jastrow factor \(F_{\rm cut}\) of
Fournais--Hoffmann-Ostenhof--Hoffmann-Ostenhof--S\o rensen, we set
\[
\phi=e^{-F_{\rm cut}}\psi .
\]
Whereas the original wave function satisfies the sharp global threshold
\(\psi\in \Bsp^s(\mathbb R^{3N})\) for every \(0\leq s<1\), we prove
that the Jastrow quotient gains one full order:
\[
\phi\in \Bsp^s(\mathbb R^{3N})
\qquad \text{for every } 0\le s<2 .
\]
The endpoint \(s=2\) is shown to be natural through an explicit
hydrogen-like eigenfunction. The many-body proof is a global
Fourier-side resolvent argument. After conjugation by the cut-off
Jastrow factor, the Coulomb singularities are converted into localized
angular coefficient blocks with admissible Fourier-control measures.
Low frequencies are controlled by the a priori \(H^1\)-bound, while
high frequencies are recovered by a Neumann fixed-point argument using
the resolvent multiplier and annular estimates for the coefficient
measures.
\end{abstract}

\tableofcontents

\section{Introduction}

\subsection{Motivation}

High-dimensional electronic structure is governed by Coulombic singularities.  In the Born--Oppenheimer approximation, an $N$-electron bound state (with $N\in \N^*$) is an eigenfunction of the operator
\begin{equation}\label{eq:intro-H}
        H=-\sum_{i=1}^N\Delta_{x^i}+V(x),
        \qquad x=(x^1,\ldots,x^N)\in\R^{3N},
\end{equation}
where for all $x=(x^1, \ldots, x^N)\in \mathbb{R}^{3N}$,
\begin{equation}\label{eq:intro-V}
        V(x)=-\sum_{i=1}^N\sum_{\ell=1}^L\frac{Z^\ell}{|x^i-R^\ell|}
        +\sum_{1\le i<j\le N}\frac1{|x^i-x^j|},
\end{equation}
where $L\in \N^*$ is the number of nuclei in a given molecule, the state of which is described by their positions $R^1, \ldots, R^L\in \mathbb{R}^3$ and their electric charges $Z^1, \ldots, Z^L >0$. Under appropriate assumptions, $H$ is a bounded from below self-adjoint operator on $L^2(\mathbb{R}^{3N})$ with domain $H^2(\mathbb{R}^{3N})$. The singularities of the potential $V$ create cusp singularities in $\psi$ at electron--nucleus and electron--electron coalescences.  These cusps are mild enough for Sobolev regularity, but they are sufficiently strong to limit Fourier $L^1$-type regularity and, consequently, Barron-space approximation properties.

The main question addressed here is whether the singular part can be removed by an explicit Jastrow factor in such a way that the remaining factor belongs to a higher spectral Barron space.  We use the cut-off Jastrow factor introduced in the sharp regularity analysis of Fournais--Hoffmann-Ostenhof--Hoffmann-Ostenhof--S\o rensen~\cite{FHOHS2005}.  The quotient
\[
        \phi=e^{-F_{\rm cut}}\psi
\]
keeps the physical cusp correction near the coalescence sets but avoids the bad behaviour at infinity caused by the non-cut-off factor.

Our main result is that the cut-off quotient satisfies
\begin{equation}\label{eq:intro-main}
        \phi\in\Bsp^s(\R^{3N})\qquad\forall s<2, 
\end{equation}
where $\Bsp^s(\R^{3N})$ denotes the spectral Barron space of order $s$ on $\R^{3N}$. This should be compared with the sharp result for the unregularized wave
function, namely \(\psi\in \Bsp^s(\R^{3N})\) for every \(s<1\) in general~\cite{Yserentant2026}.
Thus the cut-off Jastrow factor yields one full additional order in the
spectral Barron scale.

\subsection{State of the art}

The analysis of Coulombic wave functions has a long history. Kato's cusp
conditions \cite{Kato1957} identify the first-order radial behaviour at
particle coalescences, both for electron--nucleus and electron--electron
coalescence sets. A much sharper structural result was obtained by Fournais,
Hoffmann-Ostenhof, Hoffmann-Ostenhof and Sørensen, who proved that
many-electron wave functions admit a factorization by an explicit universal
factor, independent of the eigenvalue and of the particular eigenfunction,
such that the remaining factor has locally bounded second derivatives
\cite{FHOHS2005}. The same circle of ideas led to analytic representations
near simple coalescence points, of the form
\(\psi=\psi^{(1)}+|y|\psi^{(2)}\), with analytic factors when no other
collision is present \cite{FHOHS2009}.

Explicitly correlated and transcorrelated formulations exploit the same mechanism from a numerical perspective.  Flad, Hackbusch and Schneider studied best $N$-term approximations for electronic wave functions with Jastrow-type correlation factors \cite{FladHackbuschSchneider2007}.  Yserentant proved that multiplying electronic wave functions by suitable explicit correlation factors improves mixed regularity, leading to much more favourable sparse approximation estimates \cite{Yserentant2011}.  Bachmayr investigated a transcorrelated two-electron formulation and the associated hyperbolic wavelet discretization, showing that the explicitly correlated equation can improve the mixed derivative structure relevant for sparse tensor product approximation \cite{Bachmayr2012}.

In parallel, Barron spaces have become a central functional framework for
understanding dimension-robust neural network approximation. Barron's original
theorem bounds approximation by two-layer neural networks in terms of a
Fourier first moment \cite{Barron1993}. More recent works identify Barron
spaces as natural approximation spaces for two-layer neural networks and
related flow-induced models \cite{EMaWu2022}. Spectral Barron spaces,
defined by weighted \(L^1\)-norms of the Fourier transform, are particularly
well adapted to PDE regularity and Fourier-feature approximation. Chen, Lu,
Lu and Zhou developed a regularity theory for whole-space Schrödinger
equations in spectral Barron spaces \cite{ChenLuLuZhou2023}. Liao and
Ming studied approximation of spectral Barron functions by deep neural
networks, including dimension-free rates in suitable regimes
\cite{LiaoMing2025}. Recent work of Choulli, Lu and Takase develops further
functional analysis in spectral Barron spaces
\cite{ChoulliLuTakase2026}.


For Coulombic electronic wave functions themselves, Yserentant proved the sharp global result
\begin{equation}\label{eq:yserentant-intro}
        \psi\in\Bsp^s(\R^{3N})\qquad\forall s<1,
\end{equation}
for eigenvalues below the essential spectrum, and showed that the threshold cannot be improved in general, already for the hydrogen ground state \cite{Yserentant2026}.  Ming and Yu subsequently developed a Barron regularity theory for many-particle Schr\"odinger eigenfunctions under Fourier--Lebesgue assumptions on the one-particle and pairwise potentials, including singular inverse-power potentials \cite{MingYu2025}.  These results are close in spirit to the Fourier estimates used below. The new feature
of the present work is that the equation is first conjugated by a cut-off Jastrow factor.
After this conjugation, the Coulomb singularity \(1/|y|\), whose Fourier transform behaves
like \(|k|^{-2}\) in a three-dimensional collision variable \(y\), is replaced in the critical
first-order coefficients by localized angular blocks of the form
\[
  \chi(|y|)\frac{y}{|y|},
\]
whose Fourier transforms decay like \(|k|^{-3}\). This extra decay, combined with the
two powers supplied by the resolvent multiplier \((|\xi|^2+\rho)^{-1}\), is the analytic
source of the additional spectral Barron order. The global many-body argument is a Fourier-side resolvent and fixed-point argument. The
low-frequency part of \(\widehat\phi\) is controlled by the a priori
\(H^1\)-bound, while the high-frequency part is recovered as the unique fixed point of a
contraction in weighted Fourier \(L^1\) spaces. The admissible Fourier-control measures
are used to make this contraction stable under the linear collision-variable maps and
under the products appearing in \(|\nabla F_{\rm cut}|^2\).

\subsection{Contributions}

The contributions of this manuscript are the following.
\begin{enumerate}
  \item We prove that for the hydrogen ground state, the cut-off Jastrow quotient is a
  Schwartz function. Hence it belongs to \(\Bsp^s(\mathbb R^3)\) for every
  \(s\geq 0\).

  \item We prove that for the hydrogen-like eigenfunction
  \(\psi(x)=x_1 e^{-Z|x|/4}\), the cut-off quotient belongs to
  \(\Bsp^s(\mathbb R^3)\) for every \(0\leq s<2\), and not to
  \(\Bsp^2(\mathbb R^3)\). This gives the model obstruction at the endpoint.

  \item We prove the general many-body theorem by a global Fourier-side resolvent
  argument. The key analytic point is that, after extraction of the cut-off Jastrow
  factor, the critical first-order coefficients are localized angular blocks with
  borderline Fourier decay \(|k|^{-3}\) in three-dimensional collision variables.
  Together with the resolvent multiplier, this yields a high-frequency contraction in
  weighted Fourier \(L^1\) spaces.

\end{enumerate}

\section{Preliminaries}

We introduce some notation and preliminaries in this section. Section~\ref{sec:2.1} is devoted to the definition of the spectral Barron spaces and their associated norms. Section~\ref{sec:2.2} presents the cut-off Jastrow factor we consider in the present work. In all the rest of the paper, $C$ will denote some constant the value of which may change along the calculations.

\subsection{Fourier convention and spectral Barron spaces}\label{sec:2.1}

Let $d\in \mathbb{N}^*$ be an arbitrary positive integer. For any complex-valued function $u\in L^2(\R^d)$, we denote by
\begin{equation}\label{eq:fourier-convention}
        \widehat u(\xi)=\int_{\R^d}e^{-ix\cdot\xi}u(x)\dd x,
        \qquad \mbox{ for a.e. }\xi \in \R^d,
\end{equation}
so that 
$$
u(x)=(2\pi)^{-d}\int_{\R^d}e^{ix\cdot\xi}\widehat u(\xi)\dd\xi, \qquad \mbox{ for a.e. }x\in \R^d.
$$
We denote by $\mathcal S(\R^d)$ the set of Schwartz complex-valued functions defined on $\R^d$. 
For \(u\in \mathcal S'(\mathbb R^d)\), we define \(\widehat u\in\mathcal S'(\mathbb R^d)\) by
\[
  \langle \widehat u,\varphi\rangle_{\mathcal S',\mathcal S}
  =
  \langle u,\widehat\varphi\rangle_{\mathcal S',\mathcal S},
  \qquad \forall \varphi\in\mathcal S(\mathbb R^d).
\]
We also denote by $\mathcal F(u) = \widehat u$ and by $\mathcal F^{-1}$ the inverse Fourier transform.

\medskip

For any $\xi \in \R^d$, we also denote by $\ang\xi=(1+|\xi|^2)^{1/2}$.

\begin{definition}[Spectral Barron space]
For $s\ge0$, the spectral Barron space on $\R^d$ is defined as the set
\[
        \Bsp^s(\R^d)=\left\{u\in\Sscr'(\R^d):\widehat u\in L^1_{\rm loc}(\R^d),\quad
        \|u\|_{\Bsp^s}<\infty\right\},
\]
where
\[
        \|u\|_{\Bsp^s}=\int_{\R^d}\ang\xi^s|\widehat u(\xi)|\dd\xi.
\]
\end{definition}
It can be easily checked that $\Bsp^s(\R^d)$ is a Banach space for any $s\geq 0$. 

\medskip

We begin with a useful preliminary lemma which concerns the Barron regularity of Schwartz functions.

\begin{lemma}\label{lem:schwartz-barron}
If $u\in\Sscr(\R^d)$, then $u\in\Bsp^s(\R^d)$ for every $s\ge0$.
\end{lemma}

\begin{proof}
Since the Fourier transform is an automorphism of $\Sscr(\R^d)$, for every $M>0$ there exists a constant $C_M>0$ such that $|\widehat u(\xi)|\le C_M\ang\xi^{-M}$. 
Choosing \(M>d+s+1\) gives the integrability of \(\langle \xi\rangle^s|\widehat u(\xi)|\), and hence the desired result.
\end{proof}

\subsection{The cut-off Jastrow factor}\label{sec:2.2}

We detail here the cut-off Jastrow factor we consider in this work.

We consider the Hamiltonian $H$ defined in \eqref{eq:intro-H}--\eqref{eq:intro-V}. The operator $H$ is a self-adjoint bounded from below operator with domain $H^2(\R^{3N})$. Let $\chi\in C^\infty_c([0,\infty))$ satisfying
$0\leq \chi\leq 1$, $\chi(r)=1$ for all $0\leq r\leq 1$, and
$\chi(r)=0$ for all $r\geq 2$. 
  The two-body part of the Jastrow factor is given as follows: for all $x=(x^1, \ldots, x^N)\in \R^{3N}$, 
\begin{equation}\label{eq:F2cut}
        F_{2,{\rm cut}}(x)=
        -\frac12\sum_{i=1}^N\sum_{\ell=1}^L Z^\ell \chi(|x^i-R^\ell|)|x^i-R^\ell|
        +\frac14\sum_{1\le i<j\le N}\chi(|x^i-x^j|)|x^i-x^j|.
\end{equation}
Set $y^{i,\ell}=x^i-R^\ell$.  The logarithmic three-body term of Fournais--Hoffmann-Ostenhof--Hoffmann-Ostenhof--S\o rensen ~\cite{FHOHS2005} is
\begin{equation}\label{eq:F3cut}
\begin{aligned}
        F_{3,{\rm cut}}(x)
        ={}& C_0\sum_{\ell=1}^L\sum_{1\le i<j\le N}
        Z^\ell\,\chi(|y^{i,\ell}|)\chi(|y^{j,\ell}|)
        \,(y^{i,\ell}\cdot y^{j,\ell}) \\
        &\hspace{7em}\times
        \log\bigl(|y^{i,\ell}|^2+|y^{j,\ell}|^2\bigr),
\end{aligned}
\end{equation}
with $C_0=(2-\pi)/(12\pi)$.  We write
\begin{equation}\label{eq:Fcut}
        F_{\rm cut}=F_{2,{\rm cut}}+ \gamma F_{3,{\rm cut}},
        \qquad
        \phi=e^{-F_{\rm cut}}\psi,
\end{equation}
with $\gamma = 0 \mbox{ or } 1$ and where $\psi$ is an eigenfunction of the operator $H$.

The logarithmic term \(F_{3,\mathrm{cut}}\) is the cut-off version of the
three-body correction in the factorization theorem; see in particular
\cite[(1.10), (1.17)]{FHOHS2005}.

\medskip

The
Fourier Neumann argument below applies, more generally, to any bounded cut-off factor \(F\)
for which \(\nabla F\) and the scalar combination
\[
  \Delta F + |\nabla F|^2 - V
\]
are bounded and have Fourier transforms dominated by admissible
Fourier-control measures in the sense of Definition~\ref{def:admissible-measure}.

The terms in \eqref{eq:F2cut} and \eqref{eq:F3cut} are compactly supported in their collision variables, but not in the full set of variables $x\in\R^{3N}$.  Their Fourier transforms are therefore naturally measures supported on subspaces of the dual space.

\begin{remark}
The non-cut-off Jastrow factor is not suitable for a global Barron statement on $\R^{3N}$.  For the hydrogen ground state, the non-cut-off quotient is the constant function, which does not lie in $\Bsp^0(\R^3)$.  The cut-off factor removes the cusp near the origin while preserving exponential decay at infinity.
\end{remark}

\section{Two one-electron model results}\label{sec:3}

Section~\ref{sec:3} is devoted to the analysis of the two one-electron models we consider here, which both concern the electronic Hamiltonian associated to the hydrogen nucleus. 

More precisely, let
\[
        H_Z=-\Delta-\frac Z{|x|}, \qquad x\in \mathbb{R}^3, 
\]
be the bounded from below self-adjoint operator on $L^2(\mathbb{R}^3)$ with domain $H^2(\mathbb{R}^3)$. Then, the spectrum of $H_Z$ is 
$$
\sigma(H_Z) = \left\{ - \frac{Z^2}{4n^2}: \; n\in \mathbb{N}^* \right\} \cup [0, +\infty).
$$
More precisely, the essential spectrum of $H_Z$ is $\sigma_{\rm ess} = [0, +\infty)$, and each discrete eigenvalue $E_n =  - \frac{Z^2}{4n^2}$ for $n\in \mathbb{N}^*$ has multiplicity exactly equal to $n^2$. 

\medskip

In Section~\ref{sec:3.1}, we study the case of the hydrogen ground state and consider $\psi_1$ the unique (up to a multiplicative constant) eigenfunction of $H_Z$ associated to $E_1$. We prove that the associated Jastrow quotient $\phi_1$ belongs to the Schwartz space, and hence to $\Bsp^s(\R^3)$ for any $s\geq 0$.  In Section~\ref{sec:3.2}, we consider an eigenfunction $\psi_2$ of $H_Z$ associated with the eigenvalue $E_2$. We then prove that the associated Jastrow quotient $\phi_2$ belongs to $\Bsp^s(\R^3)$ for all $0\leq s <2$ but does not belong to $\Bsp^2(\R^3)$ illustrating why the value $s=2$ is an endpoint.

\subsection{The hydrogen ground state}\label{sec:3.1}

Set $a=Z/2$.  The ground state, up to a multiplicative constant and normalization, is given by
\[
        \psi_1(x)=e^{-a|x|}, \qquad x \in \R^3. 
\]
Indeed, it is easy to check that $H_Z\psi_1=-a^2\psi_1 = E_1\psi_1$.

\medskip

For one-electron systems, $N=1$, $F_{3,{\rm cut}}=0$ and
\[
        F_{\rm cut}(x)=-a\chi(|x|)|x|, \qquad x\in \R^3.
\]
Thus,
\begin{equation}\label{eq:phi1cut}
        \phi_{1}(x)=e^{-F_{\rm cut}(x)}\psi_1(x)
        =e^{-a(1-\chi(|x|))|x|}, \qquad x\in \R^3. 
\end{equation}

We then prove the following theorem. 

\begin{theorem}[Hydrogen ground state]\label{thm:hydrogen}
The function $\phi_{1}$ belongs to $\Bsp^s(\R^3)$ for all $s\geq 0$.
Actually, $\phi_{1}\in\Sscr(\R^3)$.
\end{theorem}

\begin{proof}
Let $g(x)=(1-\chi(|x|))|x|$ for $x\in \R^3$.  Since $\chi=1$ on $B(0,1)$, $g=0$ on $B(0,1)$. Moreover, since $\chi(x)=0$ for all $x\in \R^3$ such that $|x|\ge2$, $g(x)=|x|$ as soon as $|x|\ge2$.  Therefore $e^{-ag}$ is $C^\infty$ on $\R^3$, equals $1$ near the origin, and equals $e^{-a|x|}$ outside $B(0,2)$.  Every derivative is a finite sum of terms bounded by a polynomial times $e^{-a|x|}$ outside a compact set.  Hence
\[
        \sup_{x\in\R^3}\ang x^M|\partial^\beta \phi_{1}(x)|<\infty
\]
for all $M\in \mathbb{N}$ and all multi-indices $\beta \in \mathbb{N}^3$.  Thus $\phi_{1}\in\Sscr(\R^3)$, and Lemma \ref{lem:schwartz-barron} proves the claim.
\end{proof}

\subsection{A hydrogen-like state and the endpoint $s=2$}\label{sec:3.2}

Consider now the function
\[
        \psi_2(x)=x_1e^{-b|x|},\qquad x = (x_1, x_2, x_3)\in \mathbb{R}^3,
\]
with $b=Z/4$. A direct computation gives
\[
        H_Z \psi_2 = \left(-\Delta-\frac Z{|x|}\right)\psi_2=-\frac{Z^2}{16}\psi_2 = E_2 \psi_2.
\]
Thus, $\psi_2$ is an eigenfunction associated to $E_2$. With the same cut-off factor as above, and $a=Z/2$, we have
\begin{equation}\label{eq:phi2cut}
        \phi_{2}(x)=x_1e^{-b|x|+a\chi(|x|)|x|}, \qquad x\in \R^3.
\end{equation}
For all $x\in B(0,1)$, it holds that
\[
        \phi_{2}(x)=x_1e^{(a-b)|x|}=x_1e^{(Z/4)|x|}.
\]
In the local Taylor expansion of $\psi_2$ close to the origin, the first non-smooth term is the non-zero multiple
\[
  \frac Z4\,x_1|x|.
\]
Equivalently, the local expansion contains a term of the form
\(x_1|x|A(|x|)\) with \(A(0)=Z/4\neq 0\). Our aim is to prove that $\phi_2 \in \Bsp^s(\R^3)$ for all $0\leq s < 2$ but $\phi_2 \notin \Bsp^2(\R^3)$. To this aim, we first prove the following auxiliary lemma. 

\begin{lemma}[The model singularity $x_j|x|$]\label{lem:xjr}
Let $A\in C_c^\infty([0,\infty))$ and set
\[
        f_j(x)=x_j|x|A(|x|),\qquad x\in\R^3, \; j=1,2,3.
\]
Then, $f_j\in\Bsp^s(\R^3)$ for every $s<2$.  If $A(0)\ne0$, then $f_j\notin\Bsp^2(\R^3)$.
\end{lemma}

\begin{proof}
It suffices to prove the desired result for $j=1$.  For all $x=(x_1,x_2,x_3)\in \R^3$, write $f_1(x)=x_1Q(x)$ with $Q(x)=|x|A(|x|)$. 
For all $\xi \in \R^3$, let $G(\xi)=\widehat Q(\xi)$. It is then easy to check that for all $\xi\in \R^3$, $G(\xi)=G_0(\varrho)$, where
\(\varrho=|\xi|\). The radial Fourier formula in dimension three then gives
\begin{equation}\label{eq:radial}
        G_0(\varrho)=c_F\varrho^{-1}I(\varrho),
        \qquad
        I(\varrho)=\int_0^\infty h(r)\sin(\varrho r)\dd r,
        \qquad h(r)=r^2A(r), \qquad r\geq 0,\end{equation}
        with $c_F = 4 \pi$. Moreover,  $h(0)=h'(0)=0$ and $h\in C_c^\infty([0,\infty))$. 
        After two integrations by parts, we obtain
\[
  I(\varrho)
  =
  -\frac{1}{\varrho^2}
  \int_0^\infty h''(r)\sin(\varrho r)\,dr .
\]
A further integration by parts gives
\[
  \int_0^\infty h''(r)\sin(\varrho r)\,dr
  =
  \frac{h''(0)}{\varrho}+O(\varrho^{-2})
  \qquad \mbox{ as }\varrho\to+\infty.
\]
Therefore
\[
  I(\varrho)
  =
  -h''(0)\varrho^{-3}+O(\varrho^{-4})
  =
  -2A(0)\varrho^{-3}+O(\varrho^{-4}) \qquad \mbox{ as }\varrho\to+\infty.
\]
Similar calculations yield
$$
I'(\varrho)=6A(0)\varrho^{-4}+O(\varrho^{-5}) \qquad \mbox{ as } \varrho \to +\infty.
$$
Thus
\[
        G_0'(\varrho)=8c_FA(0)\varrho^{-5}+O(\varrho^{-6}),
\]
and therefore, since $\widehat{x_1Q}=i\partial_{\xi_1}G$, we obtain
\[
        \widehat f_1(\xi)=8ic_FA(0)\frac{\xi_1}{|\xi|^6}+O(|\xi|^{-6})\qquad \mbox{ as } |\xi| \to +\infty.
\]
On the one hand, this implies
\[
        |\widehat f_1(\xi)|\le C|\xi|^{-5},\qquad \mbox{for all }\xi\in \R^3 \mbox{ s.t. } |\xi|\ge1.
\]
Moreover, the function $\widehat f_1\in L^\infty(\R^3)\cap C^0(\R^3)$ since $f_1\in L^1(\R^3)$. Hence, for all $0\leq s < 2$,
\[
        \int_{\R^3}\ang\xi^s|\widehat f_1(\xi)|\dd\xi
        \le C+C\int_1^\infty r^s r^{-5}r^2\dd r
        =C+C\int_1^\infty r^{s-3}\dd r,
\]
which is finite since $s<2$.

\medskip

Assume now that $A(0)\ne0$. On the cone $\Gamma=\{\xi = (\xi_1, \xi_2, \xi_3)\in \R^3:\; |\xi_1|\ge |\xi|/2\}$, which has positive solid angle, there exist $R_0,c_0>0$ such that
\[
        |\widehat f_1(\xi)|\ge c_0|\xi|^{-5},\qquad \forall \xi\in\Gamma \; \mbox{ s.t. }|\xi|\ge R_0.
\]
Hence, there exists $c>0$ such that
\[
        \int_{\R^3}\ang\xi^2|\widehat f_1(\xi)|\dd\xi
        \ge c\int_{R_0}^\infty r^2r^{-5}r^2\dd r
        =c\int_{R_0}^\infty\frac{\dd r}{r}=\infty.
\]
Thus $f_1\notin\Bsp^2(\R^3)$.
\end{proof}

\begin{theorem}[Hydrogen-like threshold]\label{thm:hydrogenoid}
Let $\phi_{2}$ be  defined by \eqref{eq:phi2cut}. Then, 
\[
        \phi_{2}\in\Bsp^s(\R^3)\quad\text{for every }s<2,
        \qquad
        \phi_{2}\notin\Bsp^2(\R^3).
\]
\end{theorem}

\begin{proof}
Choose $\eta\in C_c^\infty(\R^3)$, radial, equal to $1$ on $B(0, 1/2)$ and supported on $\overline{B(0,1)}$.  The function $(1-\eta)\phi_2$ is smooth and exponentially decaying, hence Schwartz. Now consider the function $\eta \phi_2$ which is equal to $0$ outside $\overline{B(0,1)}$. For all $x=(x_1, x_2, x_3)\in B(0,1)$, it holds that
\[
        \eta(x)\phi_{2}(x)=\eta(x) x_1e^{c|x|},\qquad \mbox{ with }c=a-b=Z/4.
\]
We now write $x_1e^{c|x|}$ as
\[
        x_1e^{c|x|}=x_1\cosh(c|x|)+x_1\sinh(c|x|).
\]
On the one hand, the function $\cosh(c|\cdot|)$ is analytic since for all $x\in \R^3$, $\cosh(c|x|)$ is a converging power series in $|x|^2$.  Thus $\R^3 \ni x=(x_1, x_2, x_3) \mapsto \eta(x) x_1\cosh(c|x|)$ belongs to $C_c^\infty(\R^3)$ and hence to $\mathcal S(\R^3)$. 

On the other hand, for all $x\in \R^3$,
\[
        \eta(x) x_1\sinh(c|x|)=x_1|x|\eta(x) \frac{\sinh(c|x|)}{|x|}.
\]
Since $\sinh(cr)/r$ is analytic in $r^2$ at $r=0$, it holds that  $\eta(x) x_1\sinh(c|x|) = x_1|x|A(|x|)$ with $A\in C_c^\infty([0, +\infty))$ and $A(0)=c\ne0$.  Lemma~\ref{lem:xjr} then implies that the function $\R^3 \ni x\mapsto \eta(x) x_1 \sinh(c|x|)$ belongs to $\Bsp^s(\R^3)$ for all $0\leq s<2$ but does not belong to $\Bsp^2(\R^3)$. 

Thus, using Lemma~\ref{lem:schwartz-barron}, $\phi_2 \in \Bsp^s(\R^3)$ for all $0\leq s < 2$ but $\phi_2 \notin \Bsp^2(\R^3)$, and hence the desired result.  
\end{proof}

\section{A global Fourier-resolvent argument for the many-body problem}
\label{sec:global}

We now prove the main many-body result. Throughout this section we set
$d=3N$, write
$$
\phi=e^{-F_{\rm cut}}\psi,
$$
and use the Fourier convention of Section~\ref{sec:2.1}.

\begin{theorem}[Jastrow quotient in spectral Barron spaces]\label{thm} Let \(R_1,\ldots,R_L\in\mathbb R^3\) be distinct clamped nuclei with charges \(Z_\alpha>0\), and let \(H\) be the self-adjoint Friedrichs realization on \(L^2(\mathbb R^{3N})\) associated with the Coulomb Hamiltonian defined by \eqref{eq:intro-H}-\eqref{eq:intro-V}. Let \(\psi\in D(H)\) satisfy \[ H\psi=E\psi,\qquad E<\inf\sigma_{\mathrm{ess}}(H). \] Let \(F_{\rm cut}\) be the cut-off Jastrow factor defined in \eqref{eq:Fcut} avec $\gamma = 0 \mbox{ ou } 1$, and set \[ \psi=e^{F_{\rm cut}}\phi . \] Then \[ \phi\in B^s_{\rm sp}(\mathbb R^{3N}) \qquad\text{for every }0\le s<2 . \] No permutation-symmetry or spin assumption is used in the proof. \end{theorem}

\begin{remark}[Role of the spectral assumption]
The assumption \(E<\inf\sigma_{\rm ess}(H)\) places \(\psi\) in the bound-state
regime. In the proof below, after multiplication by the bounded factor
\(e^{-F_{\rm cut}}\), the only a priori information used by the low/high-frequency Neumann argument is
\[
  \phi=e^{-F_{\rm cut}}\psi\in H^1(\mathbb R^{3N}).
\]
Thus the analytic core of the argument is Proposition~\ref{prop:abstract-bootstrap}: any \(H^1\)-solution
of the conjugated equation whose coefficients are bounded and satisfy the
admissible Fourier-control assumptions belongs to \(B^s_{\rm sp}(\mathbb R^{3N})\)
for all \(0\leq s<2\). These assumptions are verified for the cut-off Jastrow
factor in Proposition~\ref{prop:coefficients}. The spectral assumption
is therefore used to justify the physical bound-state setting and the
\(H^1\) input, while the subsequent bootstrap is a global Fourier argument
and does not use a local decomposition near individual collision sets. No quantitative exponential decay estimate is used in the Fourier bootstrap; only the 
\(H^1\) a priori bound and the boundedness of the cut-off factor enter the argument.
\end{remark}


The proof occupies the rest of this section. We first derive the
conjugated equation, then verify the Fourier-control class of all
coefficients generated by $F_{\rm cut}$, and finally close the
low/high-frequency bootstrap.

\paragraph{Proof strategy.}
The proof has three main steps. First, we conjugate the eigenvalue equation by the
bounded cut-off Jastrow factor and obtain an equation of the form
\[
(-\Delta+\rho)\phi
=
2\nabla F_{\rm cut}\cdot \nabla\phi
+
\bigl(\Delta F_{\rm cut}+|\nabla F_{\rm cut}|^2-V+E+\rho\bigr)\phi .
\]
Second, we prove that every scalar coefficient appearing in this equation is bounded
and has a Fourier transform dominated by an admissible Fourier-control measure. This
step contains the cancellation of the Coulomb singularities, the treatment of the
logarithmic three-body term, and the stability of the control class under products.
Third, we use the resolvent multiplier \((|\xi|^2+\rho)^{-1}\) and the annular estimates
for admissible measures to close a high-frequency Neumann argument in weighted
Fourier \(L^1\) spaces. The low frequencies are controlled by the a priori bound
\(\phi\in H^1(\mathbb R^{3N})\).


\subsection{Fourier-control measures}

We need a convenient way to encode functions depending on only a few collision variables.  For $q\ge3$ and $m\ge0$, define
\begin{equation}\label{eq:kappa}
        \kappa_{q,m}(k)=|k|^{-2}\one_{|k|\le1}+\ang k^{-q}\bigl(1+\log\ang k\bigr)^m\one_{|k|>1},
        \qquad k\in\R^q.
\end{equation}
The singularity $|k|^{-2}$ is locally integrable in dimension $q\ge3$, and the tail $\ang k^{-q}$ is borderline at infinity.

\begin{definition}[Admissible Fourier-control measure]
\label{def:admissible-measure}
For \(q\ge 3\) and \(m\ge 0\), set
\[
 \kappa_{q,m}(k)
 =
 |k|^{-2}{\bf 1}_{|k|\le 1}
 +
 \langle k\rangle^{-q}\bigl(1+\log\langle k\rangle\bigr)^m
 {\bf 1}_{|k|>1},
 \qquad k\in\mathbb R^q .
\]
An admissible Fourier-control measure on \(\mathbb R^d\) is a finite sum of
measures of the following form. Let \(q_1,\ldots,q_p\ge 3\),
\(m_1,\ldots,m_p\ge 0\), and let
\[
   L_r:\mathbb R^{q_r}\longrightarrow \mathbb R^d,
   \qquad 1\le r\le p,
\]
be injective linear maps. We define
\[
  \mu
  =
  \mathcal L_{(L_1,\ldots,L_p)} \#
  \left(
     \bigotimes_{r=1}^p \kappa_{q_r,m_r}(k_r)\,dk_r
  \right),
\]
where the push-forward is taken by the addition map
$$
\mathcal L_{(L_1, \ldots, L_p)}: \left\{ \begin{array}{ccc}
\R^{q_1} \times \cdots \times \R^{q_p} & \to & \R^d\\
  (k_1,\ldots,k_p) &
  \mapsto & 
  \sum_{r=1}^p L_r k_r .\\
  \end{array}\right.
$$
We choose the convention that the case \(p=0\) corresponds to the measure \(\mu=\delta_0\).
\end{definition}



The following lemma is essential in the Fourier bounds we establish later on the coefficients of the conjugated equation.

\begin{lemma}[Local finiteness and annular bound for admissible measures]
\label{lem:dyadic-admissible}
Let \(\mu\) be an admissible measure in the sense of Definition~\ref{def:admissible-measure}. Then \(\mu\) is a positive Radon measure on
\(\mathbb R^d\). Moreover, there exist constants \(C>0\) and \(M\in\mathbb N\) such that,
for all \(a\in\mathbb R^d\) and all \(R\geq 1\),
\begin{equation}\label{eq:est1}
\mu\bigl(\{\zeta\in \R^d: \; R\leq |a+\zeta|<2R\}\bigr)
\leq
C
\bigl(1+\log(2+R+|a|)\bigr)^M
\min\left\{1,\left(\frac{R}{\langle a\rangle}\right)^3\right\},
\end{equation}
and
\begin{equation}\label{eq:est2}
\mu\bigl(\{\zeta\in \R^d: \; |a+\zeta|<1\}\bigr)
\leq
C
\bigl(1+\log(2+|a|)\bigr)^M
\min\left\{1,\langle a\rangle^{-3}\right\}.
\end{equation}
\end{lemma}

The proof of Lemma~\ref{lem:dyadic-admissible} makes use of the following auxiliary lemma, the proof of which is given in the appendix. 

\begin{lemma}[Dyadic scale inequality]\label{lem:aux}
Let $M_1,M_2\in \N$. For $X,Y\ge 1$, set
\[
  \alpha(X,Y):=\min\left\{1,\left(\frac{X}{Y}\right)^3\right\}.
\]
Then, there exists a constant $C=C(M_1,M_2)>0$ such that for all $R\ge 1$ and
all $A\ge 1$,
\begin{align}
&\sum_{\ell \in \mathbb{N}}
\bigl(1+\log(2+R+2^\ell)\bigr)^{M_1}
\bigl(1+\log(2+2^\ell+A)\bigr)^{M_2}
\alpha(R,2^\ell)\alpha(2^\ell,A)
\notag\\
&\hspace{4cm}\le
C\bigl(1+\log(2+R+A)\bigr)^{M_1+M_2+1}
\alpha(R,A).
\label{eq:dyadic-ineq}
\end{align}
\end{lemma}

\begin{proof}[Proof of Lemma~\ref{lem:dyadic-admissible}]
It is enough to prove the result for one elementary admissible measure; finite linear combinations with positive coefficients are
then handled by summing the estimates and increasing the constant $C$ and the logarithmic exponent $M$.

We shall use the following notation. For \(a\in\mathbb R^d\) and \(R\ge 1\), set
\[
B(a):=\{\zeta\in\mathbb R^d:\ |a+\zeta|<1\},
\qquad
A_R(a):=\{\zeta\in\mathbb R^d:\ R\le |a+\zeta|<2R\}.
\]
We also write
\[
\alpha(R,A):=\min\left\{1,\left(\frac{R}{A}\right)^3\right\},
\qquad A\ge 1,\ R\ge 1,
\]
The values of the constants $C>0$ and $M\in \mathbb{N}$ below may change along the calculations but are always independent of $a\in \mathbb{R}^d$ or $R\geq 1$. 
We now prove (\ref{eq:est1}) and (\ref{eq:est2}) for any elementary admissible measure. Any positive measure satisfying (\ref{eq:est1}) and (\ref{eq:est2}) is then obviously locally finite and is then a positive Radon measure (since by construction it is a Borel measure on $\R^d$). 

\medskip

{\bfseries Step~1:} We first treat the case $p=0$, namely when $\mu = \delta_0$. Then, 
\[
\delta_0(B(a))\neq 0
\quad\Longrightarrow\quad
|a|<1,
\]
and
\[
\delta_0(A_R(a))\neq 0
\quad\Longrightarrow\quad
R\le |a|<2R,
\]
thus (\ref{eq:est1}) and (\ref{eq:est2}) are immediate.

\medskip

{\bfseries Step~2:}
We now prove (\ref{eq:est1}) and (\ref{eq:est2}) for a single block measure $\mu$ of the form
\[
\mu = L_{\#}\bigl(\kappa_{q,m}(k)\,dk\bigr),
\qquad q\ge 3,\; m\geq 0, 
\]
where \(L:\mathbb R^q\to\mathbb R^d\) is injective. Let
$
S=L(\mathbb R^q)\subset\mathbb R^d$,
and decompose
\[
a=a_S+a_\perp,\qquad a_S\in S,\quad a_\perp\in S^\perp.
\]
Since \(L:\mathbb R^q\to S\) is an isomorphism onto its range, there are constants
\(c_L,C_L>0\) such that
\[
c_L |k|\le |Lk|\le C_L |k|,\qquad \forall k\in\mathbb R^q.
\]
Let \(k_a\in\mathbb R^q\) be uniquely defined by
\[
Lk_a=-a_S.
\]

We start with (\ref{eq:est2}). If \(B(a)\cap S=\varnothing\), there is nothing to
prove. Otherwise, the condition \(|a+Lk|<1\) implies
\[
|a_\perp|<1,
\qquad
|L(k-k_a)|<1,
\]
and hence \(k\in B(k_a,C)\). As $|a|$ goes to infinity, it holds that
\(|a_S|\simeq |a|\), and therefore \(|k_a| = \mathcal O(|a|)\). Thus, there exists $C>0$ such that for all $a\in \R^d$ such that $|a|\geq C$,
\[
\mu(B(a))
\le
C \int_{B(k_a,C)}
\langle k\rangle^{-q}\bigl(1+\log\langle k\rangle\bigr)^m\,dk
\le
C\langle a\rangle^{-q}
\bigl(1+\log\langle a\rangle\bigr)^m.
\]
Since \(q\ge 3\), this gives
\[
\mu(B(a))
\le
C\bigl(1+\log(2+|a|)\bigr)^m
\langle a\rangle^{-3}
\qquad \text{for }a\in \R^d \mbox{ s.t. } |a|\ge C.
\]
Moreover, there exists a constant $C'>0$ such that $\nu(B(a)) \leq C'$ for  all $a\in \R^d$ such that $|a|\leq C$ because
\[
\kappa_{q,m}(k)=|k|^{-2}\mathbf 1_{\{|k|\le 1\}}
+
\langle k\rangle^{-q}
\bigl(1+\log\langle k\rangle\bigr)^m\mathbf 1_{\{|k|>1\}}
\]
is locally integrable in dimension \(q\ge 3\). Hence, there exists $C>0$ and $M\in \mathbb{N}$ such that for all $a\in \R^d$
\begin{equation}\label{eq:4.16a}
\mu(B(a))
\le
C\bigl(1+\log(2+|a|)\bigr)^M
\min\{1,\langle a\rangle^{-3}\}.
\end{equation}
Hence, $\mu$ satisfies (\ref{eq:est2}). 

\medskip

We next prove (\ref{eq:est1}) for the same single block measure $\mu$. Let $R\geq 1$ and $a\in \R^d$. 

Assume first that\(\langle a\rangle\ge K_L R\), where \(K_L>0\) is a sufficiently large constant depending
only on \(L\). If \(A_R(a)\cap S=\varnothing\), there is again nothing to prove. Otherwise,
from
\[
R\le |a+Lk|<2R
\]
we get
\[
|a_\perp|<2R,
\qquad
|L(k-k_a)|<2R.
\]
Thus, there exists $C>0$ such that \(k\in B(k_a,CR)\). Since \(\langle a\rangle\ge K_L R\), choosing \(K_L\) large enough
gives that for all $k\in B(k_a, CR)$, $|k|\simeq |k_a|= \mathcal O(\langle a\rangle)$ as $\langle a \rangle$ goes to infinity. 
Therefore
\[
\mu(A_R(a))
\le
C R^q
\langle a\rangle^{-q}
\bigl(1+\log(2+R+|a|)\bigr)^m.
\]
Since \(q\ge 3\) and \(R/\langle a\rangle\le 1\) in this case, we obtain
\begin{equation}\label{eq:4.16b}
\nu(A_R(a))
\le
C
\bigl(1+\log(2+R+|a|)\bigr)^m
\left(\frac{R}{\langle a\rangle}\right)^3.
\end{equation}

It remains to consider the complementary case \(\langle a\rangle\le K_L R\). Then
\[
|a+Lk|<2R
\quad\Longrightarrow\quad
|k|\le C R.
\]
Consequently,
\[
\mu(A_R(a))
\le
C\int_{|k|\le CR}\kappa_{q,m}(k)\,dk.
\]
The singularity at \(k=0\) is integrable because \(q\ge 3\), and, for \(R\ge 1\),
\[
\int_{1\le |k|\le CR}
\langle k\rangle^{-q}
\bigl(1+\log\langle k\rangle\bigr)^m\,dk
\le
C\bigl(1+\log(2+R)\bigr)^{m+1}.
\]
Since \(\langle a\rangle\le K_LR\), we have
\[
\min\left\{1,\left(\frac{R}{\langle a\rangle}\right)^3\right\}
\ge c_L>0.
\]
Thus, after increasing the logarithmic exponent if necessary,
\begin{equation}\label{eq:4.16c}
\mu(A_R(a))
\le
C
\bigl(1+\log(2+R+|a|)\bigr)^M
\min\left\{1,\left(\frac{R}{\langle a\rangle}\right)^3\right\}.
\end{equation}
Combining~\eqref{eq:4.16b} and~\eqref{eq:4.16c} proves the annular estimate for a single block.

\medskip

{\bfseries Step~3:} We now prove that (\ref{eq:est1}) and (\ref{eq:est2}) are stable under convolution. Let \(\mu_1,\mu_2\)
be two positive Radon measures. Let us assume that for all $i=1,2$, there exist $C_i>0$ and $M_i\in \N$ such that for all $a\in \R^d$ and all $R\geq 1$,
\begin{equation}\label{eq:4.16d}
\mu_i(B(a))
\le
C_i
\bigl(1+\log(2+|a|)\bigr)^{M_i}
\min\{1,\langle a\rangle^{-3}\},
\end{equation}
and
\begin{equation}\label{eq:4.16e}
\mu_i(A_R(a))
\le
C_i
\bigl(1+\log(2+R+|a|)\bigr)^{M_i}
\alpha(R,\langle a\rangle).
\end{equation}
We claim that \(\mu_1\star\mu_2\) satisfies the same two types of estimates, with possibly larger constant and logarithmic exponent.

Using Lemma~\ref{lem:aux}, we obtain that there exists a constant $C= C(M_1, M_2)$ such that for all $R\geq 1$ and all $A\geq 1$, 

\begin{align}\notag
& \sum_{\ell \in \mathbb{N}}
\bigl(1+\log(2+R+2^\ell)\bigr)^{M_1}
\bigl(1+\log(2+2^\ell+A)\bigr)^{M_2}
\alpha(R,2^\ell)\alpha(2^\ell,A)\\\label{eq:4.16f}
&\le
C
\bigl(1+\log(2+R+A)\bigr)^{M_1+M_2+1}
\alpha(R,A).\\ \notag
\end{align}

We prove first (\ref{eq:est2}) for $\mu_1 \star \mu_2$. By Tonelli, it holds that
\[
(\mu_1\star \mu_2)(B(a))
=
\int_{\mathbb R^d}
\mu_1(B(a+\zeta))\,d\mu_2(\zeta).
\]
We decompose the \(\zeta\)-integration domain into
\[
B(a)=\{\zeta\in \R^d:\ |a+\zeta|<1\}
\]
and the dyadic annuli \(A_{2^\ell}(a)\), $\ell \in \mathbb{N}$. For $\zeta \in B(a)$,
\(\mu_1(B(a+\zeta))\le C\), while \(\mu_2(B(a))\) satisfies~\eqref{eq:4.16d}.  For $\ell \in \mathbb{N}$
estimate~\eqref{eq:4.16d} yields that for all $\zeta \in A_{2^\ell}(a)$,
\[
\mu_1(B(a+\zeta))
\le
C
\bigl(1+\log(2+2^\ell)\bigr)^{M_1}
\alpha(1,2^\ell),
\]
and estimate~\eqref{eq:4.16e} gives
\[
\mu_2(A_{2^\ell}(a))
\le
C
\bigl(1+\log(2+2^\ell+|a|)\bigr)^{M_2}
\alpha(2^\ell,\langle a\rangle).
\]
Using~\eqref{eq:4.16f} with \(R=1\), we obtain
\begin{equation}\label{eq:4.16g}
(\mu_1\star\mu_2)(B(a))
\le
C
\bigl(1+\log(2+|a|)\bigr)^M
\min\{1,\langle a\rangle^{-3}\}.
\end{equation}

We now prove (\ref{eq:est1}). Again by Tonelli,
\[
(\mu_1\star\mu_2)(A_R(a))
=
\int_{\mathbb R^d}
\mu_1(A_R(a+\zeta))\,d\mu_2(\zeta).
\]

For all $\zeta \in \R^d$ such that \(|a+\zeta|<1\), we have
\[
\mu_1(A_R(a+\zeta))
\le
C\bigl(1+\log(2+R)\bigr)^{M_1},
\]
and the mass of this region with respect to \(\mu_2\) is bounded by~\eqref{eq:4.16d}. Since
\(R\ge 1\), one has
\[
\min\{1,\langle a\rangle^{-3}\}
\le
C\,\alpha(R,\langle a\rangle),
\]
so this contribution is bounded by
\[
C
\bigl(1+\log(2+R+|a|)\bigr)^M
\alpha(R,\langle a\rangle).
\]
For all $\ell \in\mathbb{N}$ and all \(\zeta \in A_{2^\ell}(a)\), estimate~\eqref{eq:4.16e} gives
\[
\mu_1(A_R(a+\zeta))
\le
C
\bigl(1+\log(2+R+2^\ell)\bigr)^{M_1}
\alpha(R,2^\ell),
\]
while
\[
\mu_2(A_{2^\ell}(a))
\le
C
\bigl(1+\log(2+2^\ell+|a|)\bigr)^{M_2}
\alpha(2^\ell,\langle a\rangle).
\]
Summing over $\ell \in \mathbb{N}$ and using \eqref{eq:4.16f}, we obtain
\begin{equation}\label{eq:4.16h}
(\mu_1\star \mu_2)(A_R(a))
\le
C
\bigl(1+\log(2+R+|a|)\bigr)^M
\alpha(R,\langle a\rangle).
\end{equation}
Thus the class of positive measures satisfying (\ref{eq:est1}) and (\ref{eq:est2})
is stable under convolution.

\medskip

{\bfseries Step~4:} We now return to an elementary admissible measure
\[
\mu
=
\mathcal L_{(L_1,\dots,L_p)} \#
\left(
\bigotimes_{r=1}^p \kappa_{q_r,m_r}(k_r)\,dk_r
\right),
\]
where the push-forward is taken by
\[
\mathcal L_{(L_1, \ldots, L_p)}: \left\{ \begin{array}{ccc}
\R^{q_1} \times \cdots \times \R^{q_p} & \to & \R^d\\
  (k_1,\ldots,k_p) &
  \mapsto & 
  \sum_{r=1}^p L_r k_r .\\
  \end{array} \right.
\]
Actually, the result is a direct corollary of the previous single-block estimates and the convolution stability since it can easily be seen that 
$$
\mu = \nu_1 \star \cdots \star \nu_p,
$$
where for all $1\leq r \leq p$, 
\[
\nu_{r}
:=
(L_r)_{\#}
\left(
\kappa_{q_r,m_r}(k_r)\,dk_r
\right).
\]
This ends the proof of Lemma~\ref{lem:dyadic-admissible}.
\end{proof}

\begin{definition}[Domination of a Fourier transform]\label{def:domination}
Let $b\in\Sscr'(\R^d)$.  We say that $\widehat b$ is dominated by an admissible measure $\mu$ if and only if
\[
        |\langle\widehat b,\varphi\rangle_{\mathcal S', \mathcal S}|
        \le \int_{\R^d}|\varphi(\xi)|\dd\mu(\xi),
        \qquad \forall \varphi\in C_c^\infty(\R^d).
\]
\end{definition}

\begin{lemma}[Domination gives a measure]\label{lem:riesz}
If $\widehat b$ is dominated by a positive Radon measure $\mu$, then $\widehat b$ is a complex Radon measure $\nu$ satisfying $|\nu|\le\mu$.
\end{lemma}

\begin{proof}
On each compact $K \subset \R^d$, the functional $\varphi\mapsto\langle\widehat b,\varphi\rangle_{\mathcal S', \mathcal S}$ is bounded on $C_c^\infty(K)$ with respect to the uniform norm by $\mu(K)$. It can be uniquely extended by continuity to $C_c(K)$.  By the Riesz-Markov theorem, this extension is equal to integration against a complex Borel measure $\nu_K$ on $\R^d$ with $|\nu_K|\le\mu|_K$.  The measures $\nu_K$ agree on intersections by uniqueness.  Patching over an exhaustion of $\R^d$ gives a complex Radon measure $\nu$ with $|\nu|\le\mu$.
\end{proof}

\subsection{Elementary Fourier estimates for coefficient blocks}

\begin{lemma}[Localized homogeneous-logarithmic Fourier decay]
\label{lem:homogeneous-log-decay}
Let \(q\ge 1\), let \(\lambda>-q\), and let
\(a\in C^\infty(\mathbb S^{q-1})\). Let
\(\chi_0\in C^\infty_c(\mathbb R^q)\) be equal to one in a neighbourhood
of the origin. For \(\ell\in\mathbb N\), we define
\[
\forall y \in \R^q, \qquad   h(y)=\chi_0(y)|y|^\lambda
  a\!\left(\frac{y}{|y|}\right)
  \bigl(\log |y|\bigr)^\ell.
\]
Then, there exists a constant $C>0$ such that for all $k\in \R^q$, 
\begin{equation}\label{eq:esthh}
   |\widehat h(k)|
   \le
   C\langle k\rangle^{-q-\lambda}
   \bigl(1+\log\langle k\rangle\bigr)^\ell.
\end{equation}
\end{lemma}

\begin{proof} Let us first point out that $h\in L^1(\R^q)$. Thus, $\widehat{h}\in L^\infty(\R^q) \cap C^0(\R^q)$.

Choose \(\varepsilon>0\) such that \(\chi_0=1\) on \(B(0,2\varepsilon)\).
Let \(\eta\in C^\infty_C(B(0,\varepsilon))\) be equal to one in a
neighbourhood of the origin. Then \((1-\eta)h\in C^\infty_c(\mathbb R^q)\),
so it remains to study of the Fourier transform of \(\eta h\).

We can then choose \(\theta\in C^\infty_c(\{y\in \R^q: \; 1/2\le |y|\le 2\})\) such that
\[
  \sum_{n\ge 0}\theta(2^ny)=1,
  \qquad \forall y \in \R^q \mbox{ s.t. }0<|y|\le \varepsilon .
\]
Such a function \(\theta\) is obtained as follows. Let
\(\varphi\in C^\infty_c(\mathbb R^q)\) be radial, with
\(0\le \varphi\le1\), \(\varphi=1\) on \(\{y\in \R^q:\; |y|\le1\}\), and
\(\varphi=0\) on \(\{y\in \R^q: \; |y|\ge2\}\). Set
\[
  \theta(y)=\varphi(y)-\varphi(2y).
\]
Then \(\theta\in C^\infty_c(\{y \in \R^q: \; 1/2\le |y|\le2\})\). Moreover, for all $y\in \R^q\setminus\{0\}$ and all $N\in \mathbb{N}$, 
\[
  \sum_{n=0}^N\theta(2^ny)
  =
  \varphi(y)-\varphi(2^{N+1}y).
\]
Hence, if \(|y|\le1\), letting $N$ go to infinity gives
\[
  \sum_{n\ge0}\theta(2^n y)=1.
\]
After rescaling, the same construction applies on the set
\(\{y \in \R^q: \; 0<|y|\le\varepsilon\}\).

\medskip

As a consequence, we have for all $y\in \R^q$, 
\[
  \eta(y)h(y)=\sum_{n\ge0} h_n(y),
\]
where
\[
  h_n(y)
  =\eta(y)\theta(2^ny)|y|^\lambda
    a\!\left(\frac{y}{|y|}\right)(\log |y|)^\ell .
\]
By rescaling
\(y=2^{-n}z\) and integrating by parts on the fixed annulus
\(\{z\in \R^q: \; 1/2\le |z|\le 2\}\), we obtain, that, for every \(N\in \N\), there exists a constant $C_N>0$ such that for all $k\in \R^q$ and all $n\in \N$, 
\begin{equation}\label{eq:esthn}
  |\widehat h_n(k)|
  \le C_N\,2^{-n(q+\lambda)}(1+n)^\ell
        \bigl(1+2^{-n}|k|\bigr)^{-N}.
\end{equation}
Indeed, set \(y=2^{-n}z\). Then
\[
\widehat h_n(k)
=
2^{-n(q+\lambda)}
\widehat H_n(2^{-n}k),
\]
where
\[
H_n(z)
=
\eta(2^{-n}z)\theta(z)|z|^\lambda
a\!\left(\frac{z}{|z|}\right)
\bigl(\log |2^{-n}z|\bigr)^\ell .
\]
The functions \(H_n\) are supported in the fixed annulus
\(\{z\in \R^q: 1/2\le |z|\le2\}\). Moreover, for every multi-index \(\beta\in \N^q\), there exists a constant $C_\beta >0$ such that for all $n\in \N$, 
\[
\|\partial_z^\beta H_n\|_{L^\infty}
\le
C_\beta(1+n)^\ell .
\]
The only \(n\)-growth comes from the factor
\[
\log |2^{-n}z|=\log |z|-n\log 2,
\]
while all derivatives of \(\log |z|\) are bounded on the fixed annulus.
Therefore, by integrating by parts \(N\) times in the Fourier integral for
\(\widehat H_n\), we get that there exists a constant $C_N>0$ such that for all $\xi \in \R^q$ and all $n\in \N$, 
\[
|\widehat H_n(\xi)|
\le
C_N(1+n)^\ell(1+|\xi|)^{-N}.
\]
Taking \(\xi=2^{-n}k\), this yields~\eqref{eq:esthn}.

\medskip

Let now $k\in \R^q$ such that $|k|\geq 1$ and let \(J\in\mathbb N\) be such that  $2^J\le |K|<2^{J+1}$. We split the dyadic sum into the two regions \(n\le J\) and \(n>J\). We denote by \(\beta=q+\lambda>0\). From the estimate \eqref{eq:esthn}, we get
\[
  |\widehat h(k)|
  \le
  C_N
  \sum_{n\ge0}
  2^{-n\beta}(1+n)^\ell
  \bigl(1+2^{-n}|k|\bigr)^{-N}.
\]

We first consider the terms \(0\le n\le J\). Since
\(2^J\le |k|\), we have $2^{-n}|k|\ge 2^{J-n}$.Hence
\[
  \bigl(1+2^{-n}|K|\bigr)^{-N}
  \le
  C\,2^{-N(J-n)}.
\]
Therefore
\[
\begin{aligned}
  \sum_{0\le n\le J}
  2^{-n\beta}(1+n)^\ell
  \bigl(1+2^{-n}|K|\bigr)^{-N}
  &\le
  C
  \sum_{0\le n\le J}
  2^{-n\beta}(1+n)^\ell
  2^{-N(J-n)}      \\
  &=
  C
  2^{-J\beta}
  \sum_{0\le n\le J}
  2^{-(N-\beta)(J-n)}
  (1+n)^\ell .
\end{aligned}
\]
We choose \(N>\beta\). Then, after the change of
variables \(m=J-n\),
\[
\begin{aligned}
  \sum_{0\le n\le J}
  2^{-(N-\beta)(J-n)}
  (1+n)^\ell
  &=
  \sum_{m=0}^J
  2^{-(N-\beta)m}
  (1+J-m)^\ell        \\
  &\le
  (1+J)^\ell
  \sum_{m=0}^\infty
  2^{-(N-\beta)m}
  \le
  C_{\beta,N}(1+J)^\ell .
\end{aligned}
\]
Thus
\[
  \sum_{0\le n\le J}
  2^{-n\beta}(1+n)^\ell
  \bigl(1+2^{-n}|k|\bigr)^{-N}
  \le
  C\,2^{-J\beta}(1+J)^\ell .
\]

We now consider the terms \(n>J\). In this region we simply use
$\bigl(1+2^{-n}|k|\bigr)^{-N}\le 1$.
Hence
\[
\begin{aligned}
  \sum_{n>J}
  2^{-n\beta}(1+n)^\ell
  \bigl(1+2^{-n}|k|\bigr)^{-N}
  &\le
  \sum_{n>J}
  2^{-n\beta}(1+n)^\ell       \\
  &=
  2^{-J\beta}
  \sum_{m\ge1}
  2^{-m\beta}(1+J+m)^\ell ,
\end{aligned}
\]
where \(n=J+m\). Since $1+J+m\le (1+J)(1+m)$, we obtain
\[
\begin{aligned}
  \sum_{n>J}
  2^{-n\beta}(1+n)^\ell
  \bigl(1+2^{-n}|k|\bigr)^{-N}
  &\le
  2^{-J\beta}(1+J)^\ell
  \sum_{m\ge1}
  2^{-m\beta}(1+m)^\ell      \\
  &\le
  C_{\beta,\ell}\,2^{-J\beta}(1+J)^\ell ,
\end{aligned}
\]
because \(\beta>0\).

Combining the two estimates gives
\[
  |\widehat h(k)|
  \le
  C\,2^{-J\beta}(1+J)^\ell .
\]
Finally, since \(2^J\le |k|<2^{J+1}\), we have $2^{-J\beta}\le C_\beta |k|^{-\beta}$ and $1+J\le C(1+\log |k|)$.

Consequently,
\[
  |\widehat h(k)|
  \le
  C
  |k|^{-\beta}
  (1+\log |k|)^\ell .
\]
Since \(\beta=q+\lambda\), this gives
\[
  |\widehat h(k)|
  \le
  C
  |k|^{-q-\lambda}
  (1+\log |k|)^\ell,
\]
for all $k\in \R^q$ such that $|k|\geq 1$. Replacing \(|k|\) by \(\langle k\rangle\) in the right-hand side and using the fact that $\widehat{h}$ is bounded gives~\eqref{eq:esthh} for all $k\in \R^q$, up to increasing the constant $C$. 
\end{proof}

\begin{lemma}[Localized angular block]
\label{lem:angular}
Let \(q=3\), $j=1,2,3$, and for all $y= (y_1,y_2,y_3)\in \R^3$,
\[
   a_j(y)=\chi(|y|)\frac{y_j}{|y|}.
\]
Then \(\widehat a_j\in L^\infty(\R^3)\) and there exists $C>0$ such that for all $k\in \R^3$, 
\begin{equation}\label{eq:estaj}
   |\widehat a_j(k)|\le C\langle k\rangle^{-3} .
\end{equation}
Hence \(\widehat a_j\) is dominated by \(C\kappa_{3,0}(k)\,dk\).
\end{lemma}

\begin{proof}
Let $j=1,2,3$. Since \(a_j\) is bounded and compactly supported, \(a_j\in L^1(\mathbb R^3)\).
Hence \(\widehat a_j\) is bounded and continuous.

Let \(\eta\in C^\infty_c(\mathbb R^3)\) be equal to one in a neighbourhood
of the origin and supported in the set where \(\chi=1\). Then, for all $y = (y_1, y_2, y_3)\in \R^3$, 
\[
  \eta(y)a_j(y)
  = \eta(y)\frac{y_j}{|y|}
\]
is a localized homogeneous term of degree \(0\), with smooth angular factor. Lemma~\ref{lem:homogeneous-log-decay},
applied with \(q=3\), \(\lambda=0\), and \(\ell=0\), gives the existence of a constant $C>0$ such that for all $k\in \R^3$, 
\[
  |\widehat{\eta a_j}(k)|\le C\langle k\rangle^{-3}.
\]
On the other hand, \((1-\eta)a_j\in C^\infty_c(\mathbb R^3)\). Its Fourier transform thus belongs to $\mathcal S(\R^3)$. Combining the two estimates yields~\eqref{eq:estaj}, which proves the
domination by \(C\kappa_{3,0}(k)\,dk\).
\end{proof}

\begin{lemma}[Long-range Coulomb remainder]\label{lem:longrange}
For all $y\in \R^3$, let
\[
        r(y)=\frac{1-\chi(|y|)}{|y|}.
\]
Then, $\widehat r$ is dominated by $C\kappa_{3,0}(k)\dd k$.
\end{lemma}

\begin{proof}
Since \(\chi=1\) near the origin, \(r\) vanishes in a neighbourhood of
\(0\). Since \(\chi(y)=0\) as soon as $|y|\geq 2$, one has \(r(y)=|y|^{-1}\)
outside $B(0,2)$. Thus \(r\in C^\infty(\mathbb R^3)\), and, because
\(|y|^{-1}\) is harmonic away from the origin, $\Delta r\in C^\infty_c(\mathbb R^3)$. Thus \(\widehat{\Delta r}\) is a Schwartz function, 

We write
\[
  r(y)=\frac1{|y|}-\frac{\chi(|y|)}{|y|}.
\]
The Fourier transform of \(|y|^{-1}\) is a constant multiple of
\(|k|^{-2}\), while \(\chi(|y|)|y|^{-1}\in L^1(\mathbb R^3)\) has a
bounded Fourier transform. Thus, $\widehat{r} \in L^1_{\rm loc}(\R^3)$. In addition, for  all \(k\neq0\),
\[
  -|k|^2\widehat r(k)=\widehat{\Delta r}(k).
\]
As a consequence, \(\widehat r(k)\)
decays faster than any power as \(|k|\to\infty\).
Combining this rapid high-frequency decay with the low-frequency
estimate gives domination by
\(C\kappa_{3,0}(k)\,dk\).
\end{proof}

\begin{lemma}[Localized logarithmic blocks]
\label{lem:log-blocks}
Let \(Y=(y,z)\in\mathbb R^3\times\mathbb R^3\), let
\[
   N(Y)=\sqrt{|y|^2+|z|^2},\qquad P(Y)=y\cdot z,
\]
and set
\[
   G(Y)=\chi(|y|)\chi(|z|)P(Y)\log N(Y)^2.
\]
Then the Fourier transforms of
\(\partial_{Y_j}G\), \(1\le j\le 6\), and of \(\Delta_YG\), are dominated by
\[
   C\kappa_{6,m}(K)\,dK
\]
for some integer \(m\in \N\). Moreover,
\[
   \nabla_YG\in (L^\infty(\mathbb R^6))^6,
   \qquad
   \Delta_YG\in L^\infty(\mathbb R^6).
\]
\end{lemma}

\begin{proof}
It suffices to examine the joint singularity \(Y=0\), since all terms in
which a derivative falls on a cut-off are supported away from the joint
singularity and are \(C^\infty_c(\R^6)\). Near \(Y=0\),
\[
   G(Y)=P(Y)\log N(Y)^2 .
\]
Since \(P\) is homogeneous of degree \(2\),
\[
 \nabla_Y(P\log N^2)
 =
 (\nabla_YP)\log N^2
 +
 2P\frac{\cdot}{N^2}.
\]
Here \(\nabla_YP(Y)=O(|Y|)\) as $|Y|$ goes to $0$ so the first term is \(O(N(Y)|\log N(Y)|)\), and the
second term is \(O(N(Y))\). Hence \(\nabla_YG\) is bounded near the origin.

For the Laplacian, using \(\Delta_YP=0\), \(\nabla_YP=(z,y)\), and
\[
   \Delta_Y\log N^2=\frac{8}{N^2}
   \qquad\text{in dimension }6,
\]
we obtain
\[
\begin{aligned}
   \Delta_Y(P\log N^2)(Y)
   &=
   P(Y)\,\Delta_Y\log N(Y)^2
   +
   2\nabla_YP(Y)\cdot\nabla_Y\log N^2(Y)  \\
   &=
   8\frac{y\cdot z}{|y|^2+|z|^2}
   +
   4\frac{(z,y)\cdot(y,z)}{|y|^2+|z|^2}  \\
   &=
   16\frac{y\cdot z}{|y|^2+|z|^2}.
\end{aligned}
\]
This is bounded because
\[
   |y\cdot z|\le \frac12(|y|^2+|z|^2).
\]

The Fourier control follows from
Lemma~\ref{lem:homogeneous-log-decay}. Indeed, the first derivatives of
\(P\log N^2\) are finite sums of localized homogeneous-logarithmic terms of
degree \(1\) in dimension \(6\), and therefore have Fourier decay
\(\langle K\rangle^{-7}\) up to logarithmic factors. The Laplacian is a
localized homogeneous term of degree \(0\) in dimension \(6\), hence has
Fourier decay \(\langle K\rangle^{-6}\), again up to harmless logarithmic
factors. The low-frequency behaviour is harmless because all these functions
are locally integrable and compactly supported after cut-off.
\end{proof}

\begin{remark}\label{rem:log-second-derivatives}
Individual second derivatives of \(G\) are not necessarily bounded. For example,
mixed derivatives contain terms proportional to \(\log(|y|^2+|z|^2)\). Thus the
argument should not be interpreted as a \(C^{1,1}\) estimate on the logarithmic
three-body factor. The bounded quantity entering the conjugated equation is the
specific combination \(\Delta_YG\), not the full Hessian of \(G\).
\end{remark}

\begin{lemma}[Linear collision variables]\label{lem:linear} Let \(g\in\mathcal S'(\mathbb R^q)\) have Fourier transform dominated by a positive Radon measure \(\mu\) on \(\mathbb R^q\), and let \[ B:\mathbb R^d\longrightarrow \mathbb R^q \] be a surjective linear map. For \(a\in\mathbb R^q\), set \[ b(x)=g(Bx+a). \] Then \(\widehat b\) is dominated, up to a constant depending only on \(B\), by the push-forward of \(\mu\) under \(B^T\). 
\end{lemma} 

\begin{proof} 
Choose linear coordinates \(x=(y,z)\) such that \(y=Bx\). In these coordinates, \(b(x)=g(y+a)\) is independent of the tangent variable \(z\). Hence its Fourier transform is the tensor product of the Fourier transform of \(g(\cdot+a)\) in the \(y\)-variables with a Dirac mass in the dual \(z\)-variables. The translation \(a\) only produces a phase factor in Fourier variables. Returning to the original coordinates gives precisely the push-forward under \(B^T\). Taking total variations gives the claimed domination. \end{proof}

\begin{lemma}[Products of coefficient blocks]\label{lem:products}
Let
\[
  B_\ell:\mathbb R^d\longrightarrow\mathbb R^{q_\ell},
  \qquad \ell=1,2,
\]
be surjective linear maps, let \(a_\ell\in\mathbb R^{q_\ell}\), and set
\[
  b_\ell(x)=g_\ell(B_\ell x+a_\ell).
\]
Assume that \(b_\ell\in L^\infty(\mathbb R^d)\) and that \(\widehat{g_\ell}\) is dominated by a positive Radon measure \(\mu_\ell\) on \(\mathbb R^{q_\ell}\), for \(\ell=1,2\). Then the Fourier transform of \(b_1b_2\) is dominated, up to a
constant depending only on \(B_1,B_2\), by the
push-forward of \(\mu_1\otimes\mu_2\) under the map
\[
  (\eta_1,\eta_2)
  \longmapsto
  B_1^T\eta_1+B_2^T\eta_2 .
\]
In particular, products of bounded coefficient blocks whose Fourier transforms are
controlled by admissible measures are again controlled by admissible measures. 
\end{lemma}

\begin{proof}
Let \(\nu_\ell=\widehat{g_\ell}\) be the complex Radon measure associated
with \(g_\ell\), with \(|\nu_\ell|\le \mu_\ell\). The translation
\(g_\ell(\cdot+a_\ell)\) only multiplies \(\nu_\ell\) by the phase
\(e^{ia_\ell\cdot\eta_\ell}\), and hence does not change its total variation.
By Lemma~\ref{lem:linear}, the Fourier transform of
\(b_\ell(x)=g_\ell(B_\ell x+a_\ell)\) is, up to a harmless constant depending
only on \(B_\ell\), the push-forward
\[
  \widetilde\nu_\ell
  := (B_\ell^T)_\#\bigl(e^{ia_\ell\cdot\eta_\ell}\nu_\ell\bigr),
\]
and
\[
  |\widetilde\nu_\ell|\le C_\ell \widetilde\mu_\ell,
  \qquad \mbox{ with }
  \widetilde\mu_\ell:=(B_\ell^T)_\#\mu_\ell .
\]
Thus it is enough to prove the product estimate in full Fourier variables.
The desired dominating measure is, up to the same harmless constant, the
convolution \(\widetilde\mu_1 \star \widetilde\mu_2\), equivalently the push-forward
of \(\mu_1\otimes\mu_2\) by
\[
 \R^{q_1}\times \R^{q_2} \ni (\eta_1,\eta_2)\longmapsto B_1^T\eta_1+B_2^T\eta_2 \in \R^d.
\]

Choose \(\varrho\in\mathcal S(\mathbb R^d)\) with
\(\widehat\varrho(0)=1\), and choose \(\tau_R\in C^\infty_0(\mathbb R^d)\)
such that \(\tau_R=1\) on \(\{\xi\in \R^d: \; |\xi|\le R\}\), \(\tau_R=0\) on
\(\{\xi\in \R^d: \; |\xi|\ge 2R\}\), and \(0\le \tau_R\le1\). For $\varepsilon>0$, set
\[
  d\widetilde\nu_{\ell,\varepsilon,R}(\xi)
  =
  \widehat\varrho(\varepsilon\xi)\tau_R(\xi)\,
  d\widetilde\nu_\ell(\xi),
  \qquad
  b_{\ell,\varepsilon,R}
  =
  \mathcal F^{-1}\widetilde\nu_{\ell,\varepsilon,R}.
\]
The measures \(\widetilde\nu_{\ell,\varepsilon,R}\) are finite and compactly
supported. Therefore, by the Fourier inversion formula and Fubini's theorem,
\[
  \widehat{b_{1,\varepsilon,R}b_{2,\varepsilon,R}}
  =
(2\pi)^{-d}
  \Pi_\#\bigl(
    \widetilde\nu_{1,\varepsilon,R}
    \otimes
    \widetilde\nu_{2,\varepsilon,R}
  \bigr),
  \qquad
  \Pi(\xi_1,\xi_2)=\xi_1+\xi_2  \qquad \forall (\xi_1, \xi_2)\in \R^d \times \R^d.
\]
Consequently, for every \(\zeta\in C^\infty_c(\mathbb R^d)\),
\[
\begin{aligned}
 \left|
 \left\langle
   \widehat{b_{1,\varepsilon,R}b_{2,\varepsilon,R}},\zeta
 \right\rangle_{\mathcal S', \mathcal S}
 \right|
 &\le
 C
 \int_{\mathbb R^d}\int_{\mathbb R^d}
 |\zeta(\xi_1+\xi_2)|\,
 d\widetilde\mu_1(\xi_1)d\widetilde\mu_2(\xi_2).
\end{aligned}
\]
The right-hand side is finite because the corresponding push-forward measure
is locally finite. The estimate is independent of \(R\). Letting
\(R\to\infty\) and using dominated convergence on the compact set
\(\operatorname{Supp}\zeta\) gives the same inequality for
\[
  b_{\ell,\varepsilon}:=\varrho_\varepsilon*b_\ell,
  \qquad
  \varrho_\varepsilon(x)=\varepsilon^{-d}\varrho(x/\varepsilon).
\]
Finally, since \(b_\ell\in L^\infty(\mathbb R^d)\), the approximate identity
gives \(b_{\ell,\varepsilon}\to b_\ell\) in \(L^p_{\rm loc}\) for every
finite \(1 \leq p < +\infty\). The products therefore converge in \(L^1_{\rm loc}(\R^d)\), hence
in \(\mathcal D'(\mathbb R^d)\):
\[
  b_{1,\varepsilon}b_{2,\varepsilon}
  \longrightarrow b_1b_2.
\]
Passing to the limit \(\varepsilon\to0\) in the preceding inequality proves
that \(\widehat{b_1b_2}\) is dominated by the stated push-forward measure.
\end{proof}

\subsection{The conjugated equation}

In this section, let $d=3N$ and let $\psi\in H^1(\R^d)$ satisfy
\[
        H\psi=E\psi,
        \qquad E<\inf\sigma_{\rm ess}(H).
\]
Set $\psi=e^{F_{\rm cut}}\phi$ and $A=\nabla F_{\rm cut}$.  A direct computation gives
\[
        -\Delta(e^{F_{\rm cut}}\phi)
        =e^{F_{\rm cut}}\left[-\Delta\phi-2A\cdot\nabla\phi-(\Delta F_{\rm cut}+|A|^2)\phi\right].
\]
Thus
\begin{equation}\label{eq:conj}
        -\Delta\phi-2A\cdot\nabla\phi+\bigl(V-\Delta F_{\rm cut}-|A|^2-E\bigr)\phi=0.
\end{equation}
Fix $\rho>0$ and write
\begin{equation}\label{eq:rho}
        (-\Delta+\rho)\phi=2A\cdot\nabla\phi+C\phi,
        \qquad
        C=\Delta F_{\rm cut}+|A|^2-V+E+\rho.
\end{equation}

\paragraph{Distributional meaning of the conjugation.}
The computation above is understood in the sense of distributions. More
precisely, one first writes the identity for smooth regularizations of
\(F_{\mathrm{cut}}\). Let \(\theta_\varepsilon\) be a standard mollifier and,
on each fixed compact set, let \(F_\varepsilon\) denote a regularization of
\(F_{\mathrm{cut}}\). For smooth \(F_\varepsilon\) and smooth test functions
one has
\[
  -\Delta(e^{F_\varepsilon}v)
  =
  e^{F_\varepsilon}
  \Bigl(
    -\Delta v
    -2\nabla F_\varepsilon\cdot\nabla v
    -(\Delta F_\varepsilon+|\nabla F_\varepsilon|^2)v
  \Bigr).
\]
Passing to the limit \(\varepsilon\to0\) gives the same identity in
\(\mathcal D'(\mathbb R^d)\), with \(F_\varepsilon\to F_{\mathrm{cut}}\)
locally uniformly and \(\nabla F_\varepsilon\to\nabla F_{\mathrm{cut}}\)
locally in \(L^p\) for every finite \(p\). The potentially singular scalar
terms are not used separately. Instead, the distribution
\[
  \Delta F_{\mathrm{cut}} - V
\]
is identified in Proposition~\ref{prop:coefficients} below with a bounded function, after the
electron--nucleus and electron--electron Coulomb cancellations. Consequently
\[
  \Delta F_{\mathrm{cut}}+|\nabla F_{\mathrm{cut}}|^2-V+E+\rho
\]
is a bounded coefficient, and equation~\eqref{eq:rho} is an identity in
\(\mathcal D'(\mathbb R^d)\). Thus all
products in the conjugated equation may be interpreted either
distributionally or as multiplications by bounded functions after this
identification.



The boundedness and Fourier control of the coefficients \(A\) and \(C\) are
proved in Proposition~\ref{prop:coefficients}. We emphasize that no global
\(C^{1,1}\) bound on \(F_{\rm cut}\) is used. In particular, the logarithmic
three-body term has individual second derivatives with logarithmic singularities.
What is needed for the conjugated equation is only the boundedness and Fourier
control of \(\nabla F_{\rm cut}\) and of the specific scalar combination
\[
        \Delta F_{\rm cut}+|\nabla F_{\rm cut}|^2-V .
\]
This is the case when $F_{\rm cut}= F_{2, \rm cut}$ and when $F_{\rm cut} = F_{2, \rm cut} + F_{3, \rm cut}$.

Let us emphasize here that, since \(F_{\rm cut}\) and \(\nabla F_{\rm cut}\) are bounded, and since
\(\psi\in H^1(\mathbb R^d)\), we have
\[
        \phi=e^{-F_{\rm cut}}\psi\in H^1(\mathbb R^d).
\]

\medskip

\begin{proposition}[Fourier control of the conjugated coefficients]
\label{prop:coefficients}
Each scalar component of
\[
        A=\nabla F_{\rm cut},
        \qquad
        C=\Delta F_{\rm cut}+|\nabla F_{\rm cut}|^2-V+E+\rho
\]
has Fourier transform dominated by an admissible measure. Moreover
\[
        A,C\in L^\infty(\mathbb R^{3N}).
\]
\end{proposition}

For the proof of Proposition~\ref{prop:coefficients}, we make use of the following lemma, the proof of which is given in the appendix.

\begin{lemma}[Distributional Laplacian of the cut-off cusp]\label{lem:distrib}
Let \(\chi\in C^\infty_c([0,\infty))\) be equal to one in a neighbourhood
of the origin, and set
\[
  f(y)=\chi(|y|)|y|,\qquad y\in\mathbb R^3.
\]
Then \(f\in W^{1,\infty}_{\mathrm{loc}}(\mathbb R^3)\), and its
distributional gradient is represented by the locally bounded vector field
\[
  \nabla_y f(y)
  =
  \chi(r)\frac{y}{r}+\chi'(r)y,
  \qquad r=|y|.
\]
Moreover, the distributional Laplacian of \(f\) is represented by the
locally integrable function
\[
  \Delta_y f(y)
  =
  \frac{2\chi(r)}{r}
  +
  4\chi'(r)
  +
  r\chi''(r),
  \qquad r=|y|.
\]
In particular, no Dirac mass occurs at the origin.
\end{lemma}

\begin{proof}[Proof of Proposition~\ref{prop:coefficients}]
We inspect the two-body and three-body parts separately.

For $y\in \R^3$, let \(r=|y|\) and, using Lemma~\ref{lem:distrib}, the following identities hold in the distributional sense: 
\[
        f(y)=\chi(r)r,\qquad y\in\mathbb R^3.
\]
Then
\[
        \nabla_y f(y)
        =
        \chi(r)\frac{y}{r}+\chi'(r)y,
\]
and
\[
        \Delta_y f(y)
        =
        \frac{2\chi(r)}{r}+4\chi'(r)+r\chi''(r).
\]
The first term in \(\nabla_y f\) is the localized angular block of
Lemma~\ref{lem:angular}, while \(\chi'(r)y\) is \(C^\infty_c(\R^3)\). Thus every component of
\(\nabla_y f\) has admissible Fourier control.

\medskip

For an electron-nucleus variable \(y=x^i-R^\ell\) (with $1\leq i \leq N$ and $1\leq \ell \leq L$), the corresponding term in
\(F_{2,\rm cut}\) is
\[
        -\frac{Z^\ell}{2}\chi(|y|)|y|.
\]
Therefore near \(y=0\),
\[
        \Delta_{x^i}\left(-\frac{Z^\ell}{2}\chi(|y|)|y|\right)
        =
        -Z^\ell\frac{\chi(|y|)}{|y|}
        + b_{\ell i}(y),
\]
where \(b_{\ell i}\in L^\infty\) is smooth and supported where the derivatives
of \(\chi\) occur. Since the electron--nucleus part of \(V\) is
\[
        -\frac{Z^\ell}{|x^i-R^\ell|},
\]
the singular terms cancel in
\[
        \Delta F_{2,\rm cut}-V.
\]
The remaining electron-nucleus contribution is a finite sum of bounded smooth
cutoff terms and of long-range remainders of the form
\[
        \frac{1-\chi(|y|)}{|y|},
\]
which are handled by Lemma~\ref{lem:longrange}.

For an electron-electron variable \(y=x^i-x^j\), the corresponding term in
\(F_{2,\rm cut}\) is
\[
        \frac14\chi(|y|)|y|.
\]
Since
\[
        (\Delta_{x^i}+\Delta_{x^j})g(x^i-x^j)
        =
        2\Delta_y g(y),
\]
we obtain near \(y=0\)
\[
\begin{aligned}
        (\Delta_{x^i}+\Delta_{x^j})
        \left(\frac14\chi(|y|)|y|\right)
        &=
        \frac12\Delta_y(\chi(|y|)|y|)                                      \\
        &=
        \frac{\chi(|y|)}{|y|}
        + b_{ij}(y),
\end{aligned}
\]
where \(b_{ij}\in L^\infty\) is supported away from \(y=0\). This cancels the
electron-electron singularity \(1/|x^i-x^j|\) in \(V\). The remaining terms are
again bounded cutoff terms and long-range remainders controlled by Lemma~\ref{lem:longrange}.

It remains to treat the logarithmic three-body part. For one summand, write
\[
        y=x^i-R^\ell,\qquad z=x^j-R_\ell,\qquad
        N(y,z)^2=|y|^2+|z|^2,\qquad P(y,z)=y\cdot z.
\]
The singular core is
\[
        P\log N^2 .
\]
By Lemma~\ref{lem:log-blocks},
\[
\nabla_{(y,z)}
\bigl(P\log N^2\bigr)\in L^\infty(\mathbb R^6),
\qquad
\Delta_{(y,z)}
\bigl(P\log N^2\bigr)\in L^\infty(\mathbb R^6),
\]
and their Fourier transforms are dominated by admissible measures in the six-dimensional
collision variable \((y,z)\). More explicitly,
\[
\Delta_{(y,z)}
\bigl(P\log N^2\bigr)
=
16\,\frac{y\cdot z}{|y|^2+|z|^2},
\]
which is bounded. Terms in which at least one derivative falls on one of the cutoffs
\(\chi(|y|)\) or \(\chi(|z|)\) are supported away from the joint singularity
\((y,z)=(0,0)\) or are compactly supported smooth functions. They therefore have
rapidly decaying Fourier transforms and are also dominated by admissible measures.
After the linear substitution
\[
\R^{3N} \ni x=(x^1, \ldots, x^N)\longmapsto (x^i-R^\ell,x^j-R^\ell)\in \R^6,
\]
Lemma~\ref{lem:linear} gives admissible Fourier control in the full
configuration variable \(x\in\mathbb R^{3N}\).

We have therefore proved admissible Fourier control and boundedness for
\(\nabla F_{\rm cut}\) and for the scalar combination
\[
  \Delta F_{\rm cut}-V.
\]
The long-range Coulomb remainders occur in this scalar Laplacian--potential
combination. They do not arise from the quadratic coefficient
\(|\nabla F_{\rm cut}|^2\), which is treated separately below by product
stability of the Fourier-control class.

\medskip

Finally, we treat the quadratic coefficient \(|\nabla F_{\rm cut}|^2\).
Each scalar component of \(\nabla F_{\rm cut}\) is a finite linear combination of
bounded coefficient blocks of the following three types:
\[
g(x^i-R^\ell),\qquad
g(x^i-x^j),\qquad
g(x^i-R^\ell,x^j-R^\ell).
\]
For the first two types, \(g\) is a two-body block in \(\mathbb R^3\), consisting of a
localized angular term and smooth compactly supported remainders. These are controlled
by Lemma~\ref{lem:angular}, with the harmless smooth terms absorbed into the
same admissible class. For the long-range Coulomb remainders occurring in
\(\Delta F_{\rm cut}-V\), the control is given by Lemma~\ref{lem:longrange}.
For the third type, \(g\) is a logarithmic three-body block in \(\mathbb R^6\), controlled
by Lemma~\ref{lem:log-blocks}. In all cases, the passage from collision variables
to the full variable \(x\) is handled by Lemma~\ref{lem:linear}.

Thus every product appearing in \(|\nabla F_{\rm cut}|^2\) is a finite linear combination
of products of the form
\[
b_1(x)b_2(x)
=
g_1(B_1x+a_1)\,g_2(B_2x+a_2),
\]
where \(B_\ell:\mathbb R^{3N}\to\mathbb R^{q_\ell}\) is surjective,
\(q_\ell\in\{3,6\}\), \(a_\ell\in\mathbb R^{q_\ell}\), and each \(b_\ell\) is bounded
with Fourier transform dominated by an admissible measure in the corresponding
collision variable. The possible products are summarized as follows:
\[
\begin{array}{c|c}
\text{product type} & \text{collision dimensions} \\
\hline
\mathrm{EN}\times\mathrm{EN} & 3+3 \\
\mathrm{EN}\times\mathrm{EE} & 3+3 \\
\mathrm{EE}\times\mathrm{EE} & 3+3 \\
\mathrm{EN}\times F_3 & 3+6 \\
\mathrm{EE}\times F_3 & 3+6 \\
F_3\times F_3 & 6+6 .
\end{array}
\]
Here \(\mathrm{EN}\) denotes an electron--nucleus gradient block,
\(\mathrm{EE}\) an electron--electron gradient block, and \(F_3\) a
logarithmic three-body gradient block. Lemma~\ref{lem:products} implies that the Fourier
transform of each such product is dominated by the push-forward of the product of the
two controlling measures under
\[
\R^{q_1}\times \R^{q_2}\ni (\eta_1,\eta_2)\longmapsto B_1^T\eta_1+B_2^T\eta_2 \in \R^{3N}.
\]
This push-forward is admissible in the sense of Definition~\ref{def:admissible-measure},
and its local finiteness follows from Lemma~\ref{lem:dyadic-admissible}. Hence
\(|\nabla F_{\rm cut}|^2\) has admissible Fourier control.

The constant \(E+\rho\) contributes only a multiple of \(\delta_0\), which is
admissible. Therefore every component of \(A\) and the scalar coefficient
\(C\) are bounded and have admissible Fourier control.
\end{proof}

\begin{remark}[Checklist of the coefficient verification]
The proof of Proposition~\ref{prop:coefficients} verifies the following five points.
\begin{enumerate}
  \item The electron--nucleus Coulomb singularities are cancelled by the
  electron--nucleus part of \(\Delta F_{2,\rm cut}\).

  \item The electron--electron Coulomb singularities are cancelled by the
  electron--electron part of \(\Delta F_{2,\rm cut}\), taking into account the
  identity
  \[
    (\Delta_{x^i}+\Delta_{x^j})g(x^i-x^j)=2\Delta_y g(y).
  \]

  \item The remaining long-range Coulomb terms are of the form
  \[
    \frac{1-\chi(|y|)}{|y|}
  \]
  and are controlled by Lemma~\ref{lem:longrange}.

  \item The logarithmic three-body factor contributes bounded terms to
  \(\nabla F_{\rm cut}\) and to the scalar Laplacian combination entering the
  conjugated equation. No uniform bound on the full Hessian of the logarithmic
  term is used.

  \item The quadratic coefficient \(|\nabla F_{\rm cut}|^2\) is controlled by
  the product stability of the admissible Fourier-control class, namely
  Lemma~\ref{lem:products} together with the local finiteness and annular bounds of
  Lemma~\ref{lem:dyadic-admissible}.
\end{enumerate}
\end{remark}

\subsection{Weighted convolution estimate}

Let $\mu$ be an admissible measure.  

\begin{lemma}[Resolvent-convolution bound]
\label{lem:resolvent-convolution}
Let $\rho>0$, \(0\le s<2\), and let \(\mu\) be an admissible measure. Then there exist
\(C_s>0\) and \(M\in\mathbb N\) such that, for all \(\eta\in\mathbb R^d\),
\[
   \int_{\mathbb R^d}
   \frac{\langle\xi\rangle^s}{|\xi|^2+\rho}\,
   d\mu(\xi-\eta)
   \le
   C_s\langle\eta\rangle^{s-2}
   \bigl(1+\log\langle\eta\rangle\bigr)^M
   +C_s{\bf 1}_{|\eta|\le 2}.
\]
Consequently, there exists $C_s>0$ and $M\in \N$ such that for all \(\Lambda\ge 2\) and all $\eta \in \R^d$ with \(|\eta|\ge \Lambda/2\),
\[
   \int_{|\xi|>\Lambda}
   \frac{\langle\xi\rangle^s}{|\xi|^2+\rho}\,
   d\mu(\xi-\eta)
   \le
   C_s\Lambda^{-2}
   \bigl(1+\log\Lambda\bigr)^M
   \langle\eta\rangle^s ,
\]
and
\[
   \int_{|\xi|>\Lambda}
   \frac{\langle\xi\rangle^s\langle\eta\rangle}{|\xi|^2+\rho}\,
   d\mu(\xi-\eta)
   \le
   C_s\Lambda^{-1}
   \bigl(1+\log\Lambda\bigr)^M
   \langle\eta\rangle^s .
\]
\end{lemma}

\begin{proof}
Let $\eta \in \R^d$. By Lemma~\ref{lem:dyadic-admissible},
\[
\mu\bigl(\{\zeta\in \R^d:\; |\eta+\zeta|<1\}\bigr)
\leq
C
\bigl(1+\log(2+|\eta|)\bigr)^M
\min\{1,\langle\eta\rangle^{-3}\}.
\]
Therefore
\[
\int_{|\xi|<1}
\frac{\langle\xi\rangle^s}{|\xi|^2+\rho}\,d\mu(\xi-\eta)
\leq
C\,\mathbf 1_{|\eta|\leq2}
+
C
\langle\eta\rangle^{-3}
\bigl(1+\log\langle\eta\rangle\bigr)^M
\mathbf 1_{|\eta|>2}.
\]
Since \(0\leq s<2\), the second term is bounded by
\[
C
\langle\eta\rangle^{s-2}
\bigl(1+\log\langle\eta\rangle\bigr)^M.
\]

Define now, for all $R\geq 1$,  $A_R(\eta)=\{\zeta:\ R\le |\eta+\zeta|<2R\}$. For all $\ell \in \N$, let us denote by $R_\ell:= 2^\ell$. It then holds that
\[
   \int_{\mathbb R^d}
   \frac{\langle\xi\rangle^s}{|\xi|^2+\rho}\,
   d\mu(\xi-\eta)
   \le
   C\sum_{\ell \in \N} R_\ell^{s-2}\mu(A_{R_\ell}(\eta))
   +C{\bf 1}_{|\eta|\le 2}.
\]
If $\ell\in\N$ is such that \(R_\ell\le \langle\eta\rangle/2\), Lemma~\ref{lem:dyadic-admissible} gives
\[
   \mu(A_{R_\ell}(\eta))
   \le
   C(1+\log\langle\eta\rangle)^M
   \left(\frac{R_\ell}{\langle\eta\rangle}\right)^3 .
\]
Therefore
\[
\sum_{\ell\in \N: \; R_\ell\le \langle\eta\rangle/2}
R_\ell^{s-2}\mu(A_{R_\ell}(\eta))
\le
C(1+\log\langle\eta\rangle)^M
\langle\eta\rangle^{-3}
\sum_{\ell\in \N: \; R_\ell\le \langle\eta\rangle/2}R_\ell^{s+1}
\le
C_s\langle\eta\rangle^{s-2}
(1+\log\langle\eta\rangle)^M .
\]
If $\ell \in \N$ is such that \(R_\ell\ge \langle\eta\rangle/2\), Lemma~\ref{lem:dyadic-admissible} gives
\[
   \mu(A_{R_\ell}(\eta))
   \le C(1+\log R_\ell)^M,
\]
and since \(0\leq s<2\),
\[
\sum_{\ell \in \N: \; R_\ell\ge \langle\eta\rangle/2}
R_\ell^{s-2}\mu(A_{R_\ell}(\eta))
\le
C_s\langle\eta\rangle^{s-2}
(1+\log\langle\eta\rangle)^M .
\]
This proves the first estimate.

We now restrict the previous sums to the set of integers $\ell\in \N$ such that \(R_\ell\ge \Lambda\). For the zero-order
estimate, the preceding argument gives
\[
   \int_{|\xi|>\Lambda}
   \frac{\langle\xi\rangle^s}{|\xi|^2+\rho}\,
   d\mu(\xi-\eta)
   \le
   C_s\langle\eta\rangle^{s-2}
   (1+\log\langle\eta\rangle)^M .
\]
If \(|\eta|\ge \Lambda/2\), then, after increasing \(C_s\),
\[
   \langle\eta\rangle^{s-2}
   (1+\log\langle\eta\rangle)^M
   \le
   C_s\Lambda^{-2}(1+\log\Lambda)^M
   \langle\eta\rangle^s .
\]
Indeed, this is equivalent to
\[
   (1+\log\langle\eta\rangle)^M
   \le
   C_M(1+\log\Lambda)^M
   \left(\frac{\langle\eta\rangle}{\Lambda}\right)^2,
\]
which follows from the elementary inequality
\begin{equation}\label{eq:elementary}
   (1+\log x)^M\le C_{M,\varepsilon}x^\varepsilon,
   \qquad x\ge 1,
\end{equation}
with \(\varepsilon=2\).

For the first-order estimate, multiplying the sum by
\(\langle\eta\rangle\) gives
\[
   \int_{|\xi|>\Lambda}
   \frac{\langle\xi\rangle^s\langle\eta\rangle}{|\xi|^2+\rho}\,
   d\mu(\xi-\eta)
   \le
   C_s\langle\eta\rangle^{s-1}
   (1+\log\langle\eta\rangle)^M .
\]
Again, if \(|\eta|\ge \Lambda/2\),
\[
   \langle\eta\rangle^{s-1}
   (1+\log\langle\eta\rangle)^M
   \le
   C_s\Lambda^{-1}(1+\log\Lambda)^M
   \langle\eta\rangle^s ,
\]
because logarithmic growth is absorbed by the polynomial factor
\(\langle\eta\rangle/\Lambda\) using again~\eqref{eq:elementary} with $\varepsilon = 1$. This proves the two high-frequency
estimates.
\end{proof}

\subsection{Operators associated with the conjugated equation}

For a coefficient \(b\) whose Fourier transform is the complex Radon measure
\(\nu_b=\widehat b\), dominated by an admissible measure \(\mu_b\), it holds that, for all $f\in \mathcal S(\R^d)$, 
\[
  \widehat{bf}
  =
 (2\pi)^{-d} \nu_b\star \widehat f,
  \qquad
  \widehat{b\,\partial_j f}
  =(2\pi)^{-d} \nu_b\star(i\xi_j\widehat f).
\]
Accordingly, we formally define, for $u\in \mathcal S(\R^d)$ and $\rho>0$,
\[
  (T_{b,0}u)(\xi)
  =
  \frac{(2\pi)^{-d}}{|\xi|^2+\rho}
  \int_{\mathbb R^d}u(\eta)\,d\nu_b(\xi-\eta),
\]
and
\[
  (T_{b,1,j}u)(\xi)
  =
  \frac{(2\pi)^{-d}}{|\xi|^2+\rho}
  \int_{\mathbb R^d} i\eta_j u(\eta)\,d\nu_b(\xi-\eta).
\]
These are the Fourier representations of
\((-\Delta+\rho)^{-1}(bf)\) and
\((-\Delta+\rho)^{-1}(b\,\partial_j f)\) for $u = \widehat{f}$.

\medskip

For complete rigor, the operators are first defined by truncating the
coefficient measure. For all $R\geq 1$, let
\[
  \nu_b^R := \mathbf 1_{\{|\zeta|\leq R\}}\nu_b .
\]
and set
\[
  (T^R_{b,0}u)(\xi)
  =
  \frac{(2\pi)^{-d}}{|\xi|^2+\rho}
  \int_{\mathbb R^d} u(\eta)\,d\nu_b^R(\xi-\eta),
\]
and
\[
  (T^R_{b,1,j}u)(\xi)
  =
  \frac{(2\pi)^{-d}}{|\xi|^2+\rho}
  \int_{\mathbb R^d} i\eta_j u(\eta)\,d\nu_b^R(\xi-\eta).
\]
The estimates below are first proved for \(T^R_{b,0}\) and \(T^R_{b,1,j}\),
with constants independent of \(R\). The operators \(T_{b,0}\) and
\(T_{b,1,j}\) are then defined as the corresponding limits in the stated
target spaces. The agreement with the physical distributional operators
\[
  f\mapsto (-\Delta+\rho)^{-1}(bf),
  \qquad
  f\mapsto (-\Delta+\rho)^{-1}(b\,\partial_j f),
\]
follows by testing against Schwartz functions and passing to the vague limit
\(\nu_b^R\to\nu_b\).

\medskip

For \(t\ge 0\), we write
\[
   L^1_t
   =
   L^1(\mathbb R^d,\langle\xi\rangle^t\,d\xi),
   \qquad
   \|v\|_{L^1_t}
   =
   \int_{\mathbb R^d}\langle\xi\rangle^t|v(\xi)|\,d\xi,
\]
and
\[
   L^2_t
   =
   L^2(\mathbb R^d,\langle\xi\rangle^{2t}\,d\xi),
   \qquad
   \|v\|_{L^2_t}
   =
   \left(
   \int_{\mathbb R^d}\langle\xi\rangle^{2t}|v(\xi)|^2\,d\xi
   \right)^{1/2}.
\]
We also write $L^1 = L^1_0 = L^1(\R^d)$ and $L^2 = L^2_0  $

Also, for any $\Lambda\geq 2$, let \(P_\Lambda\) denote the multiplication operator by
\({\bf 1}_{\{|\xi|>\Lambda\}}\), and set \(Q_\Lambda=I-P_\Lambda\). Then, we have the following result.

\begin{lemma}[Operator bounds and distributional interpretation]
\label{lem:operator-bounds}
Let $\rho>0$, \(b\in L^\infty(\mathbb R^d)\), and assume that \(\widehat b\) is the
complex Radon measure \(\nu_b\), with \(|\nu_b|\le \mu_b\), where \(\mu_b\)
is admissible. Define initially
for any Schwartz function $u\in \mathcal S(\R^d)$ and for all $\xi\in \R^d$,
\[
   (T_{b,0}u)(\xi)
   =
   \frac{(2\pi)^{-d}}{|\xi|^2+\rho}
   \int_{\mathbb R^d} u(\eta)\,d\nu_b(\xi-\eta),
\]
and
\[
   (T_{b,1,j}u)(\xi)
   =
   \frac{(2\pi)^{-d}}{|\xi|^2+\rho}
   \int_{\mathbb R^d} i\eta_j u(\eta)\,d\nu_b(\xi-\eta).
\]
Let \(0\le s<2\). Then there are constants \(C_s>0\) and
\(M\in\mathbb N\), depending on \(b,s,\rho\) and on the controlling measure,
such that for all \(\Lambda\ge 2\) and for all $v\in L^1_{\rm loc}(\R^d)$ such that $P_\Lambda v\in L^1_s$,
\[
   \|P_\Lambda T_{b,0}P_\Lambda v\|_{L^1_s}
   \le
   C_s\Lambda^{-2}(1+\log\Lambda)^M
   \|P_\Lambda v\|_{L^1_s},
\]
and
\[
   \|P_\Lambda T_{b,1,j}P_\Lambda v\|_{L^1_s}
   \le
   C_s\Lambda^{-1}(1+\log\Lambda)^M
   \|P_\Lambda v\|_{L^1_s}.
\]
Moreover, for each fixed \(\Lambda\ge 2\),
\[
   P_\Lambda T_{b,0}Q_\Lambda:L^1\to L^1_s,
   \qquad
   P_\Lambda T_{b,1,j}Q_\Lambda:L^1\to L^1_s
\]
are bounded operators, with constants depending on \(\Lambda\).

The same formulae agree in \(\mathcal S'(\mathbb R^d)\) with the Fourier transform of the (physical)
distributional operators
\[
   \mathcal S'(\R^d) \ni f\longmapsto (-\Delta+\rho)^{-1}(bf)\in \mathcal S'(\R^d),
   \qquad
    \mathcal S'(\R^d) \ni  f\longmapsto (-\Delta+\rho)^{-1}(b\,\partial_j f)\in \mathcal S'(\R^d).
\]
\end{lemma}

\begin{proof}
For any Schwartz \(f\in \mathcal S(\R^d)\), it holds that
\[
   \widehat{bf}
   =
(2\pi)^{-d}\nu_b\star\widehat f
\]
in \(\mathcal S'(\mathbb R^d)\). Similarly,
\[
   \widehat{b\,\partial_j f}
   =
 (2\pi)^{-d}\nu_b \star (i\xi_j\widehat f).
\]
Multiplication by \((|\xi|^2+\rho)^{-1}\) gives exactly the two formulae
above.

Taking total variations in the truncated formula, using
\(|\nu_b^R|\leq \mu_b\), and applying Tonelli's theorem, we obtain estimates
with constants independent of \(R\). Passing to the limit \(R\to\infty\)
gives the stated bounds for \(T_{b,0}\) and \(T_{b,1,j}\), that is: there exists $C>0$ such that for all $v\in \mathcal S (\R^d)$,
\[
\begin{aligned}
   \|P_\Lambda T_{b,0}P_\Lambda v\|_{L^1_s}
   &\le
   C
   \int_{|\eta|>\Lambda}|v(\eta)|
   \int_{|\xi|>\Lambda}
   \frac{\langle\xi\rangle^s}{|\xi|^2+\rho}
   \,d\mu_b(\xi-\eta)\,d\eta  \\
   &\le
   C_s\Lambda^{-2}(1+\log\Lambda)^M
   \int_{|\eta|>\Lambda}\langle\eta\rangle^s|v(\eta)|\,d\eta,
\end{aligned}
\]
where the last step is Lemma~\ref{lem:resolvent-convolution}. This proves the
zero-order estimate. The first-order estimate is identical, using
\(|\eta_j|\le\langle\eta\rangle\) and the second high-frequency estimate of
Lemma~\ref{lem:resolvent-convolution}.

Since \(Q_\Lambda v\) is supported in \(\{\eta\in \R^d: \; |\eta|\le \Lambda\}\), then for all $\eta$ belonging to the support of \(Q_\Lambda v\), 
\(\langle\eta\rangle\le C_\Lambda\) for some constant $C_\Lambda>0$ which depends on $\Lambda$. The same Tonelli argument, now without
requiring a small high-frequency factor, gives boundedness of
\(P_\Lambda T_{b,0}Q_\Lambda\) and \(P_\Lambda T_{b,1,j}Q_\Lambda\) from
\(L^1\) to \(L^1_s\), with constants depending on \(\Lambda\).

Finally, since \(b\in L^\infty(\R^d)\), multiplication by \(b\) maps \(L^2\) to
\(L^2\), and the operator \(b\partial_j\) maps \(H^1(\R^d)\) to \(L^2\) in the
distributional sense. Applying \((-\Delta+\rho)^{-1}\), and using density of
Schwartz functions in the relevant spaces, identifies the Fourier-side
operators with the physical distributional operators in
\(\mathcal S'(\mathbb R^d)\).
\end{proof}

\subsection{Low/high-frequency Neumann argument}

\begin{proposition}[Abstract Barron regularity result]
\label{prop:abstract-bootstrap}
Let \(d\geq 1\), let \(\rho>0\), and let
\(A_1,\ldots,A_d,C\in L^\infty(\mathbb R^d)\). Assume that the Fourier
transforms of \(A_1,\ldots,A_d\) and \(C\) are complex Radon measures whose
total variations are dominated by admissible Fourier-control measures.
Let \(\phi\in H^1(\mathbb R^d)\) solve, in \(\mathcal D'(\mathbb R^d)\),
\[
  (-\Delta+\rho)\phi
  =
  2\sum_{j=1}^d A_j\partial_j\phi + C\phi .
\]
Then
\[
  \phi\in \Bsp^s(\mathbb R^d)
  \qquad\text{for every }0\leq s<2.
\]
\end{proposition}

\begin{proof}
Let $u=\widehat\phi$ and let $0\leq s < 2$. Our aim is to prove that $u\in L^1_s$ which will give the desired result. 

For \(\Lambda\geq 2\), we set
\[
  X_\Lambda
  :=
  \left\{
    v\in L^2_1\cap L^1_s
    \;:\;
    \operatorname{Supp}\; v\subset\{\xi\in\mathbb R^d:\ |\xi|>\Lambda\}
  \right\},
\]
equipped with the norm
\[
  \|v\|_{X_\Lambda}
  :=
  \|v\|_{L^2_1}+\|v\|_{L^1_s}.
\]
Equivalently, $X_\Lambda=P_\Lambda(L^2_1\cap L^1_s)$. All functions in \(X_\Lambda\) are Fourier-side
functions supported in the high-frequency region \(\{\xi\in \R^d: \; |\xi|>\Lambda\}\).

\medskip

The Fourier form of the conjugated equation is
\[
   u=Tu,
\]
where
\[
   T
   =
   T_{C,0}
   +
   2\sum_{j=1}^d T_{A_j,1,j}.
\]

\medskip

Since \(\phi\in H^1(\mathbb R^d)\), its Fourier transform
$u=\widehat\phi$ belongs to \(L^2_1\). Hence, for each fixed \(\Lambda\geq2\),
\[
Q_\Lambda u\in L^1_t
\qquad\text{for every }t\geq0,
\]
because \(Q_\Lambda u\) is supported in the bounded ball \(\{\xi\in \R^d: \; |\xi|\leq\Lambda\}\) and
Cauchy's inequality gives
\[
\|Q_\Lambda u\|_{L^1_t}
\leq
\left(
\int_{|\xi|\leq\Lambda}
\langle\xi\rangle^{2t-2}\,d\xi
\right)^{1/2}
\|u\|_{L^2_1}.
\]

For any
\(v\in X_\Lambda\), we write $f=\mathcal F^{-1}v$. Let \(b\in L^\infty(\mathbb R^d)\). For all $\xi \in \R^d$ such that $|\xi|>\Lambda$, it holds that
\[
\frac{\langle\xi\rangle}{|\xi|^2+\rho}
\leq
C\Lambda^{-1}.
\]
Therefore,
\[
\|P_\Lambda T_{b,0}v\|_{L^2_1}
=
\left\|
P_\Lambda\mathcal F\bigl((-\Delta+\rho)^{-1}(bf)\bigr)
\right\|_{L^2_1}
\leq
C\Lambda^{-1}\|b\|_{L^\infty}\|f\|_{L^2}
=
C\Lambda^{-1}\|b\|_{L^\infty}\|v\|_{L^2}.
\]
Similarly,
\[
\|P_\Lambda T_{b,1,j}v\|_{L^2_1}
=
\left\|
P_\Lambda\mathcal F\bigl((-\Delta+\rho)^{-1}(b\partial_j f)\bigr)
\right\|_{L^2_1}
\leq
C\Lambda^{-1}\|b\|_{L^\infty}\|\partial_j f\|_{L^2}
\leq
C\Lambda^{-1}\|b\|_{L^\infty}\|v\|_{L^2_1}.
\]
Consequently,
\[
\|P_\Lambda T P_\Lambda v\|_{L^2_1}
\leq
C\Lambda^{-1}\|P_\Lambda v\|_{L^2_1}.
\]

Next, using Lemma~\ref{lem:operator-bounds}, we obtain:
\[
   \|P_\Lambda T P_\Lambda v\|_{L^1_s}
   \le
   C_s\Lambda^{-1}(1+\log\Lambda)^M\|v\|_{L^1_s}.
\]
Combining the \(L^2_1\)- and \(L_s^1\)-estimates, we obtain
\[
  \|P_\Lambda T P_\Lambda v\|_{X_\Lambda}
  \leq
  C_s\Lambda^{-1}(1+\log\Lambda)^M\|v\|_{X_\Lambda}.
\]
Choosing \(\Lambda\) sufficiently large, we may assume that
\[
S_\Lambda := P_\Lambda T P_\Lambda
\]
is a contraction on \(X_\Lambda\). More precisely, choosing $\Lambda$ sufficiently large, there exists
\(0<\theta<1\) such that
\begin{equation}\label{eq:4.19b}
\|S_\Lambda v\|_{L^2_1}
\le
\theta_2 \|v\|_{L^2_1},
\qquad
\forall v\in P_\Lambda L^2_1,
\end{equation}
\begin{equation}\label{eq:4.19c}
\|S_\Lambda v\|_{L^1_s}
\le
\theta \|v\|_{L^1_s},
\qquad
\forall v\in P_\Lambda L^1_s,
\end{equation}
and
\begin{equation}\label{eq:4.19a}
\|S_\Lambda v\|_{X_\Lambda}
\le
\theta_X \|v\|_{X_\Lambda},
\qquad \forall v\in X_\Lambda.
\end{equation}

Set
\[
g_\Lambda := P_\Lambda T Q_\Lambda u .
\]
As observed above, $g_\Lambda \in X_\Lambda$. Hence the Neumann series
\[
w
=
\sum_{n=0}^\infty S_\Lambda^n g_\Lambda
\]
converges in \(X_\Lambda\), and its sum is the unique element of \(X_\Lambda\)
satisfying
\begin{equation}\label{eq:4.19d}
w = S_\Lambda w + g_\Lambda .
\end{equation}

We now identify this abstract fixed point with the true high-frequency part of \(u\).
Since \(u\) solves \(u=Tu\) in \(\mathcal S'(\mathbb R^d)\), we have
\begin{equation}\label{eq:4.19e}
P_\Lambda u
=
S_\Lambda P_\Lambda u + g_\Lambda
\qquad\text{in } \mathcal S'(\mathbb R^d).
\end{equation}
At this stage \(P_\Lambda u\) is known a priori only to belong to \(L^2_1\), not to
\(L^1_s\). Therefore the comparison with \(w\) is made in the \(L^2_1\)-topology.
Indeed, since \(w\in X_\Lambda\subset L^2_1\), the difference
\[
z := P_\Lambda u - w
\]
belongs to \(P_\Lambda L^2_1\), and \eqref{eq:4.19d}--\eqref{eq:4.19e} imply
\[
z = S_\Lambda z
\qquad\text{in } \mathcal S'(\mathbb R^d).
\]
Both sides belong to \(L^2_1\), hence the equality holds in \(L^2_1\). Using
\eqref{eq:4.19b}, we obtain
\[
\|z\|_{L^2_1}
=
\|S_\Lambda z\|_{L^2_1}
\le
\theta \|z\|_{L^2_1}.
\]
Since \(\theta<1\), it follows that \(z=0\). Thus
\[
P_\Lambda u = w \in X_\Lambda \subset L^1_s .
\]
Together with the low-frequency estimate \(Q_\Lambda u\in L^1_s\), this proves
\[
u\in L^1_s,
\]
or equivalently
\[
\phi\in \Bsp^s(\mathbb R^d).
\]
This proves Proposition~\ref{prop:abstract-bootstrap}.
\end{proof}

We are now in a position to prove Theorem~\ref{thm} which is a direct corollary of the previous results.

\begin{proof}[Proof of Theorem~\ref{thm}]
Theorem~\ref{thm} follows by
applying Proposition~\ref{prop:abstract-bootstrap} with \(d=3N\), \(A=\nabla F_{\rm cut}\), and
\(C=\Delta F_{\rm cut}+|\nabla F_{\rm cut}|^2-V+E+\rho\), using
Proposition~\ref{prop:coefficients}.
\end{proof}

\section{Conclusion}

The cut-off Jastrow factor improves the global spectral Barron regularity of Coulombic electronic eigenfunctions.  For the unregularized wave function, the Coulomb kernel has Fourier behaviour $|k|^{-2}$ in each three-dimensional collision variable, which leads to the sharp threshold $s<1$.  After conjugation by the cut-off Jastrow factor, the singular first-order coefficients are angular and have Fourier behaviour $|k|^{-3}$.  The resolvent of $-\Delta+\rho$ then provides the additional smoothing needed to reach every $s<2$.

The model hydrogen-like state shows that $s=2$ is a natural endpoint obstruction in general, although the many-body theorem itself only asserts membership for $s<2$: the residual singularity $x_j|x|$ has Fourier decay $|\xi|^{-5}$ in dimension three, giving a logarithmic divergence in the $\Bsp^2$ norm. Thus the result \(\phi\in B^s_{\rm sp}\) for all \(s<2\) is consistent
with the model obstruction at the endpoint \(s=2\), and should be understood as the
natural global regularity statement for the Jastrow-regularized factor.

\section{Appendix}

\subsection{Proof of Lemma~\ref{lem:aux}}

\begin{proof}[Proof of Lemma~\ref{lem:aux}]
Let $R\geq 1$ and $A\geq 1$. For all $\ell \in \N$, we write $S_\ell = 2^\ell$,
\[
  L_{R,S_\ell}:=1+\log(2+R+S_\ell),\qquad
  L_{S_\ell,A}:=1+\log(2+S_\ell+A),
\]
and
\[
  L_{R,A}:=1+\log(2+R+A).
\]
Throughout the proof, the constants may depend on $M_1$ and $M_2$, but never
on $R$, $A$, or $\ell$ (i.e. $S$).

\medskip

The left-hand side of \eqref{eq:dyadic-ineq} then reads as
\[
  \sum_{\ell \in \N} L_{R,S_\ell}^{M_1}L_{S_\ell,A}^{M_2}
  \alpha(R,S_\ell)\alpha(S_\ell,A).
\]
We distinguish two cases.

\medskip
\noindent\textbf{Case 1: $A\le 2R$.}
In this case $R/A\ge 1/2$, and therefore
\[
  \alpha(R,A)=\min\left\{1,\left(\frac{R}{A}\right)^3\right\}\ge 2^{-3}.
\]
It is therefore enough to prove
\[
  \sum_{\ell \in \N} L_{R,S_\ell}^{M_1}L_{S_\ell,A}^{M_2}
  \alpha(R,S_\ell)\alpha(S_\ell,A)
  \lesssim L_{R,A}^{M_1+M_2+1}.
\]
We split the sum into the two sets of indices $\ell \in \N$ such that $S_\ell\le 2R$ and $S_\ell>2R$.

First assume $S_\ell\le 2R$. Since we have
\[
  \alpha(R,S_\ell)\alpha(S_\ell,A)\le 1,
\]
and
\[
  2+R+S_\ell\le 2+3R,
\]
and, since $A\le 2R$ and $R\ge 1$,
\[
  2+3R\le 3(2+R+A), 
\]
we obtain
\[
  L_{R,S_\ell}\le 1+\log\bigl(3(2+R+A)\bigr)
  \le C L_{R,A}.
\]
Similarly,
\[
  2+S_\ell+A\le 2+2R+A\le 2+4R\le 4(2+R+A),
\]
and hence
\[
  L_{S_\ell,A}\le C L_{R,A}.
\]
Thus, for $S_\ell\le 2R$,
\[
  L_{R,S_\ell}^{M_1}L_{S_\ell,A}^{M_2}\le C L_{R,A}^{M_1+M_2}.
\]
The number of integers $\ell \in \N$ such that $S_\ell\in[1,2R]$ is bounded by
\[
  C(1+\log(2+R)).
\]
Since $1+\log(2+R)\le L_{R,A}$, we get
\begin{align*}
  \sum_{\ell \in \N: \; S_\ell\le 2R}
  L_{R,S_\ell}^{M_1}L_{S_\ell,A}^{M_2}
  \alpha(R,S_\ell)\alpha(S_\ell,A)
  &\le C L_{R,A}^{M_1+M_2}\bigl(1+\log(2+R)\bigr) \\
  &\le C L_{R,A}^{M_1+M_2+1}.
\end{align*}

Now assume $S_\ell>2R$. Then
\[
  \alpha(R,S_\ell)=\left(\frac{R}{S_\ell}\right)^3,
  \qquad
  \alpha(S_\ell,A)\le 1.
\]
Because $A\le 2R<S_\ell$, we have
\[
  2+R+S_\ell\le 2+2S_\ell,
  \qquad
  2+S_\ell+A\le 2+2S_\ell,
\]
and therefore
\[
  L_{R,S_\ell}^{M_1}L_{S_\ell,A}^{M_2}
  \le C\bigl(1+\log(2+S_\ell)\bigr)^{M_1+M_2}.
\]
Consequently,
\begin{align*}
&\sum_{\ell\in \N: \; S_\ell>2R}
  L_{R,S_\ell}^{M_1}L_{S_\ell,A}^{M_2}
  \alpha(R,S_\ell)\alpha(S_\ell,A) \\
&\hspace{2cm}\le
  C\sum_{\ell\in \N: \; S_\ell>2R}
  \left(\frac{R}{S_\ell}\right)^3
  \bigl(1+\log(2+S_\ell)\bigr)^{M_1+M_2}.
\end{align*}
Let $\ell_0\in \N$ be the lowest integer such that $S_{\ell_0}$ is strictly larger than $2R$. Then
$2R<S_{\ell_0}\le 4R$, and for all $\ell \geq \ell_0$, there exists $n\in \N$ such that $S_\ell=2^nS_{\ell_0}$. Hence
\[
  \left(\frac{R}{S_\ell}\right)^3
  =\left(\frac{R}{S_{\ell_0}}\right)^3 2^{-3n}\le C2^{-3n}.
\]
Also,
\[
  1+\log(2+2^nS_{\ell_0})
  \le C\bigl(1+\log(2+R)+n\bigr).
\]
Using the elementary bound
\[
  \sum_{n\ge 0}2^{-3n}(B+n)^M\le C_M B^M,
  \qquad B\ge 1,
\]
with $B=1+\log(2+R)$ and $M=M_1+M_2$, we obtain
\[
  \sum_{\ell \in \N: \; S_\ell>2R}
  \left(\frac{R}{S_\ell}\right)^3
  \bigl(1+\log(2+S_\ell)\bigr)^{M_1+M_2}
  \le C\bigl(1+\log(2+R)\bigr)^{M_1+M_2}.
\]
Since $1+\log(2+R)\le L_{R,A}$, this contribution is bounded by
$C L_{R,A}^{M_1+M_2}$, and therefore by
$C L_{R,A}^{M_1+M_2+1}$. This concludes the proof in the case $A\le 2R$.

\medskip
\noindent\textbf{Case 2: $A>2R$.}
In this case
\[
  \alpha(R,A)=\left(\frac{R}{A}\right)^3.
\]
We split the sum over $\ell \in \N$ into three sets of indices $\ell$ such that
\[
  S_\ell\le R,
  \qquad R<S_\ell\le A,
  \qquad S_\ell>A.
\]

First assume $S_\ell\le R$. Then
\[
  \alpha(R,S_\ell)=1,
  \qquad
  \alpha(S_\ell,A)=\left(\frac{S_\ell}{A}\right)^3,
\]
and therefore
\[
\alpha(R,S_\ell)\alpha(S_\ell,A)=\left(\frac{S_\ell}{A}\right)^3.
\]
Moreover $S_\ell\le R<A$, so
\[
  2+R+S_\ell\le 2+2R\le 2+R+A,
\]
and
\[
  2+S_\ell+A\le 2+R+A.
\]
Hence
\[
  L_{R,S_\ell}^{M_1}L_{S_\ell,A}^{M_2}\le L_{R,A}^{M_1+M_2}.
\]
Thus
\begin{align*}
&\sum_{\ell\in \N: \; S_\ell\le R}
  L_{R,S_\ell}^{M_1}L_{S_\ell,A}^{M_2}
  \alpha(R,S_\ell)\alpha(S_\ell,A) \\
&\hspace{2cm}\le
  L_{R,A}^{M_1+M_2}A^{-3}
  \sum_{\ell \in \N: \; S_\ell\le R} S_\ell^3.
\end{align*}
The geometric sum satisfies
\[
  \sum_{\ell\in \N: \; S_\ell\le R} S_\ell^3\le C R^3.
\]
Therefore
\[
  \sum_{\ell\in \N: \; S_\ell\le R}
  L_{R,S_\ell}^{M_1}L_{S_\ell,A}^{M_2}
  \alpha(R,S_\ell)\alpha(S_\ell,A)
  \le C L_{R,A}^{M_1+M_2}\left(\frac{R}{A}\right)^3.
\]
This is bounded by the desired right-hand side.

Next assume $R<S_\ell\le A$. Then
\[
  \alpha(R,S_\ell)=\left(\frac{R}{S_\ell}\right)^3,
  \qquad
  \alpha(S_\ell,A)=\left(\frac{S_\ell}{A}\right)^3.
\]
Thus the product is independent of $S_\ell$:
\[
  \alpha(R,S_\ell)\alpha(S_\ell,A)=\left(\frac{R}{A}\right)^3.
\]
Also,
\[
  2+R+S_\ell\le 2+R+A,
\]
and
\[
  2+S_\ell+A\le 2+2A\le 2(2+R+A),
\]
so
\[
  L_{R,S_\ell}^{M_1}L_{S_\ell,A}^{M_2}\le C L_{R,A}^{M_1+M_2}.
\]
It remains to count the number of integers $\ell \in \N$ satisfying $R<S_\ell\le A$. This number is at
most
\[
  C\bigl(1+\log(A/R)\bigr).
\]
Since $R\ge 1$, we have $A/R\le A$, and hence
\[
  1+\log(A/R)\le 1+\log A\le C L_{R,A}.
\]
Therefore
\begin{align*}
&\sum_{\ell\in \N: \; R<S_\ell\le A}
  L_{R,S_\ell}^{M_1}L_{S_\ell,A}^{M_2}
  \alpha(R,S_\ell)\alpha(S_\ell,A) \\
&\hspace{2cm}\le
  C L_{R,A}^{M_1+M_2}
  \left(\frac{R}{A}\right)^3
  \bigl(1+\log(A/R)\bigr) \\
&\hspace{2cm}\le
  C L_{R,A}^{M_1+M_2+1}
  \left(\frac{R}{A}\right)^3.
\end{align*}
This is exactly the desired bound in the intermediate region.

Finally assume that $\ell\in \N$ is such that $S_\ell>A$. Then $S_\ell>A>R$, and hence
\[
  \alpha(R,S_\ell)=\left(\frac{R}{S_\ell}\right)^3,
  \qquad
  \alpha(S_\ell,A)=1.
\]
Moreover,
\[
  2+R+S_\ell\le 2+2S_\ell,
  \qquad
  2+S_\ell+A\le 2+2S_\ell,
\]
so
\[
  L_{R,S_\ell}^{M_1}L_{S_\ell,A}^{M_2}
  \le C\bigl(1+\log(2+S_\ell)\bigr)^{M_1+M_2}.
\]
Thus
\begin{align*}
&\sum_{\ell\in \N: \; S_\ell>A}
  L_{R,S_\ell}^{M_1}L_{S_\ell,A}^{M_2}
  \alpha(R,S_\ell)\alpha(S_\ell,A) \\
&\hspace{2cm}\le
  C R^3
  \sum_{\ell\in \N: \; S_\ell}
S_\ell^{-3}\bigl(1+\log(2+S_\ell)\bigr)^{M_1+M_2}.
\end{align*}
Let $\ell_0$ be the lowest integer such that $S_{\ell_0}$ is strictly larger than $A$. Then
$A<S_{\ell_0}\le 2A$, and every dyadic $S_\ell>A$ can be written as $S_\ell=2^nS_{\ell_0}$ for some $n\in \N$.
Hence
\[
  S_\ell^{-3}=2^{-3n}S_{\ell_0}^{-3}\le C2^{-3n}A^{-3},
\]
and
\[
  1+\log(2+2^nS_{\ell_0})
  \le C\bigl(1+\log(2+A)+n\bigr).
\]
Using again
\[
  \sum_{n\in \N}2^{-3n}(B+n)^M\le C_M B^M,
  \qquad B\ge 1,
\]
we get
\[
  \sum_{\ell\in \N: \; S_\ell>A}
S_\ell^{-3}\bigl(1+\log(2+S_\ell)\bigr)^{M_1+M_2}
  \le C A^{-3}\bigl(1+\log(2+A)\bigr)^{M_1+M_2}.
\]
Consequently,
\[
  \sum_{\ell \in\N: \; S_\ell>A}
  L_{R,S_\ell}^{M_1}L_{S_\ell,A}^{M_2}
  \alpha(R,S_\ell)\alpha(S_\ell,A)
  \le C\left(\frac{R}{A}\right)^3
  \bigl(1+\log(2+A)\bigr)^{M_1+M_2}.
\]
Since
\[
  1+\log(2+A)\le L_{R,A},
\]
this is bounded by
\[
  C L_{R,A}^{M_1+M_2}
  \left(\frac{R}{A}\right)^3,
\]
and hence by the desired right-hand side.

Combining the three regions $S_\ell\le R$, $R<S_\ell\le A$, and $S_\ell>A$ proves
\eqref{eq:dyadic-ineq} when $A>2R$. Together with the first case, the lemma follows.
\end{proof}

\subsection{Proof of Lemma~\ref{lem:distrib}}

\begin{proof}[Proof of Lemma~\ref{lem:distrib}]
Away from the origin the computation is classical. If
\[
  F(r)=\chi(r)r,
\]
then, for \(r>0\),
\[
  F'(r)=\chi'(r)r+\chi(r),
  \qquad
  F''(r)=\chi''(r)r+2\chi'(r).
\]
Hence, for \(y\neq0\),
\[
  \nabla_y f(y)=F'(r)\frac{y}{r}
  =
  \chi(r)\frac{y}{r}+\chi'(r)y,
\]
and, using the radial Laplacian in dimension three,
\[
\begin{aligned}
  \Delta_y f(y)
  &=
  F''(r)+\frac2rF'(r)  \\
  &=
  r\chi''(r)+2\chi'(r)
  +
  \frac2r\bigl(r\chi'(r)+\chi(r)\bigr) \\
  &=
  \frac{2\chi(r)}{r}+4\chi'(r)+r\chi''(r).
\end{aligned}
\]
It remains to justify that this identity holds in the sense of distributions
on all of \(\mathbb R^3\), without an additional term supported at the origin.

Let \(\varphi\in C^\infty_c(\mathbb R^3)\). Choose \(\varepsilon>0\) small
enough so that \(\chi(r)=1\) for \(0\le r\le 2\varepsilon\). We apply Green's
formula on
\[
  \Omega_\varepsilon=\{y\in\mathbb R^3:\ |y|>\varepsilon\}.
\]
Since \(\varphi\) is compactly supported, there is no boundary contribution at
infinity. Thus
\[
  \int_{\Omega_\varepsilon} f\,\Delta\varphi\,dy
  =
  \int_{\Omega_\varepsilon} (\Delta f)_{\mathrm{cl}}\varphi\,dy
  +
  \int_{|y|=\varepsilon}
  \bigl(
    f\,\partial_n\varphi-\varphi\,\partial_n f
  \bigr)\,dS,
\]
where \((\Delta f)_{\mathrm{cl}}\) denotes the classical Laplacian away from
the origin, and \(n\) is the outward normal to \(\Omega_\varepsilon\). On the
inner boundary \(|y|=\varepsilon\), this outward normal is
\[
  n=-\frac{y}{|y|}.
\]
Since \(\chi=1\) near the origin, one has \(f(y)=|y|=\varepsilon\) and
\(\nabla f(y)=y/|y|\) on \(|y|=\varepsilon\). Hence
\[
  \partial_n f
  =
  \nabla f\cdot n
  =
  -1,
  \qquad
  \partial_n\varphi
  =
  -\partial_r\varphi .
\]
Therefore the boundary term is
\[
\begin{aligned}
  \int_{|y|=\varepsilon}
  \bigl(
    f\,\partial_n\varphi-\varphi\,\partial_n f
  \bigr)\,dS
  &=
  \int_{|y|=\varepsilon}
  \bigl(
    -\varepsilon\,\partial_r\varphi+\varphi
  \bigr)\,dS .
\end{aligned}
\]
This term tends to zero as \(\varepsilon\to0\). Indeed,
\[
  \left|
  \int_{|y|=\varepsilon}
  \varepsilon\,\partial_r\varphi\,dS
  \right|
  \le
  4\pi\varepsilon^3\|\nabla\varphi\|_{L^\infty},
\]
and
\[
  \left|
  \int_{|y|=\varepsilon}
  \varphi\,dS
  \right|
  \le
  4\pi\varepsilon^2\|\varphi\|_{L^\infty}.
\]
Consequently,
\[
  \langle \Delta f,\varphi\rangle
  =
  \lim_{\varepsilon\to0}
  \int_{|y|>\varepsilon}
  (\Delta f)_{\mathrm{cl}}(y)\varphi(y)\,dy.
\]
Since
\[
  (\Delta f)_{\mathrm{cl}}(y)
  =
  \frac{2\chi(r)}{r}+4\chi'(r)+r\chi''(r),
  \qquad r=|y|,
\]
and since the only singular term near the origin is \(2/r\), which belongs to
\(L^1_{\mathrm{loc}}(\mathbb R^3)\), we may pass to the limit and obtain
\[
  \langle \Delta f,\varphi\rangle
  =
  \int_{\mathbb R^3}
  \left(
    \frac{2\chi(r)}{r}+4\chi'(r)+r\chi''(r)
  \right)
  \varphi(y)\,dy.
\]
This proves the distributional Laplacian formula.

The formula for the distributional gradient follows similarly, or directly
from the fact that \(f\in W^{1,\infty}_{\mathrm{loc}}(\mathbb R^3)\) and the
classical gradient formula holds almost everywhere away from the origin. The
possible singularity at the single point \(y=0\) is irrelevant for the weak
gradient.

From the approximation viewpoint, the result identifies the Jastrow
quotient, rather than the original electronic wave function, as the
natural object carrying higher Fourier \(L^1\)-type regularity. This is
consistent with explicitly correlated and transcorrelated approaches:
the explicit cusp factor removes the leading Coulomb obstruction, while
the remaining factor has the spectral Barron regularity expected to be
more favourable for dimension-robust Fourier-feature and neural-network
approximation schemes.
\end{proof}

\paragraph{Acknowledgements} The author acknowledges the financial support of European Research Council (ERC) under the European Union’s
Horizon 2020 Research and Innovation Programme – Grant Agreement n°101077204 HighLEAP. ChatGPT 5.4 Pro was used for some parts of the proof.

\end{document}